\documentclass[11pt]{amsart} 
\usepackage{verbatim, latexsym, amssymb, amsmath,color}
\usepackage{epsfig}

\newcommand{\fint}{\mathop{\int\makebox(-15,2){\rule[4pt]{.9em}{0.6pt}}\kern-4pt}\nolimits}

\def\XXint#1#2#3{{\setbox0=\hbox{$#1{#2#3}{\int}$}
     \vcenter{\hbox{$#2#3$}}\kern-.5\wd0}}

\def\R{\mathbb R}
\def\RP{\mathbb {RP}}
\def\N{\mathbb N}
\def\Z{\mathbb Z}

\def\H{\mathbb H}

\def\dist{\mathrm{dist}}

\newtheorem{thm}{Theorem}[section]
\newtheorem{lemm}{Lemma}[section]
\newtheorem{cor}{Corollary}[section]
\newtheorem{prop}{Proposition}[section]
\theoremstyle{remark}

\theoremstyle{definition}

\title[]{Conformal currents and the entropy of negatively curved three-manifolds}
\author{Fernando C. Marques and Andr\'e Neves}
\address{Princeton University \\ Fine Hall \\ Princeton NJ 08544 \\ USA}
\email{coda@math.princeton.edu}
\address{University of Chicago \\ Department of Mathematics \\ Chicago IL 60637\\ USA}
\email{aneves@math.uchicago.edu}
\thanks{The first author is partly supported by NSF-DMS-2105557 and a Simons Investigator Grant. The second author is partly supported by NSF-DMS-2005468 and a Simons Investigator Grant.}

\begin{document}

\maketitle

\begin{abstract}
 In this paper, we  describe the intersection  between geodesic and conformal currents  on closed hyperbolic three-manifolds.

We use this to prove some sharp bounds which involve the Liouville entropy of a negatively curved metric, the minimal surface entropy, and the area ratio.  Using these ideas we also give a new proof  of    the Mostow Rigidity Theorem in the three-dimensional case. 
\end{abstract}

\section{Introduction}

Let $\Gamma$ be a  group of isometries which acts  properly discontinuously on the  hyperbolic space $\H^3$, so that $M=\H^3/\Gamma$ is a closed manifold.

 In \cite{bonahon}, Bonahon defined
geodesic currents as   $\Gamma$-invariant Radon measures on the space of  hyperbolic geodesics and, in the two-dimensional case, he defined  the intersection number of geodesic currents. These concepts are  essential objects in the study of geodesics in 
surfaces of negative curvature.  In \cite{labourie}, Labourie defined conformal currents as $\Gamma$-invariant Radon measures  on the space of quasi-circles of $S^2$. 

In this paper, we will use the intersection number between conformal and geodesic currents and derive some applications.

We denote  the space of geodesic currents by $\mathcal{C}(\Gamma)$  and that of  conformal currents by $\mathcal{C}_2(\Gamma)$. The hyperbolic metric on $M$  is denoted by $\overline{g}$ and $g$ denotes a metric   of negative sectional curvature on $M$.

A nontrivial closed curve $\gamma$ in $M$ induces a homotopy class $[\gamma]$ and $l_g([\gamma])$ denotes the least length of curves in $[\gamma]$. If $\Sigma$ is an essential (or incompressible)
surface  in $M$ then it  induces a homotopy class $\Pi$. The surface $\Sigma$ is said to be quasi-Fuchsian if the asymptotic boundary of its lift to $\H^3$ is a quasi-circle. The group $\Gamma$ acts on $S^2=\partial \H^3$ by conformal transformations, and hence preserve the space of quasi-circles. We will denote by ${\rm area}_g(\Pi)$ the least area of surfaces in $\Pi$.

 A map $\Pi \mapsto \delta_\Pi\in {\mathcal{C}_2}(\Gamma)$ was defined in \cite[Section 5.2]{labourie} for homotopy classes $\Pi$ of quasi-Fuchsian surfaces which generalizes the map $[\gamma] \mapsto \delta_{[\gamma]}\in \mathcal{C}(\Gamma)$ of \cite{bonahon}.  There is a Radon measure $\nu$, with support the space of round circles 
of $S^2$, that is invariant by the group $\text{\rm M\"{o}b}(S^2)$ of $\text{\rm M\"{o}bius}$ transformations. These are the maps induced
on $S^2$ by the hyperbolic isometries and hence $\nu\in {\mathcal{C}_2}(\Gamma)$. The negatively curved metric $g$  induces a Liouville current $\lambda_g\in \mathcal{C}(\Gamma)$. These examples of geodesic currents and conformal currents are described in detail in Section \ref{integral.geometry}  and Section \ref{examples.currents}.

The intersection between geodesic and conformal currents can be defined following \cite{bonahon}, and it is a bilinear map
$$I=I_\Gamma:{\mathcal{C}_2}(\Gamma)\times \mathcal C(\Gamma)\rightarrow \mathbb{R}_+.$$
This map was also used in \cite{brody-reyes} to approximate by geometric actions of $\Gamma$  on cube complexes the action of $\Gamma$ on the  hyperbolic space. 

The intersection form $I$ is related to the geometric intersection number.  For instance, $I(\delta_{\Pi},\delta_{[\gamma]})\leq \#(\Sigma\cap \alpha),$
where $\Sigma$ is any immersed surface in $\Pi$ and $\alpha$ is any closed curve in $[\gamma]$ which intersect transversely. If $\Sigma \in \Pi$ is a  totally geodesic surface and $\alpha\in [\gamma]$ is a closed geodesic for the hyperbolic metric which intersect transversely, then $I(\delta_{\Pi},\delta_{[\gamma]})= \#(\Sigma\cap \alpha)$. If $\Sigma$ and $\alpha$ are induced by the immersions $i_\Sigma:\Sigma\rightarrow M$, $i_\alpha:S^1\rightarrow M$, then 
$\# (\Sigma \cap \alpha)$ is the number of  $(x,y)\in \Sigma \times S^1$ such that $i_\Sigma(x)=i_\alpha(y)$. 

We prove:
 \begin{thm}\label{intersection.number.thm} 
For  a negatively curved metric $(M,g)$:
 \begin{itemize}
\item[(i)] $I(\delta_{\Pi},\lambda_g)\leq \pi\, {\rm area}_g(\Pi),$ and if equality holds then $\Pi$ contains a totally geodesic surface in the metric $g$;
\item[(ii)] $ I(\nu,\lambda_g)=2\pi^2{\rm vol}_g(M)i(g,\overline{g})$, where $i(g,\overline{g})$ denotes the geodesic stretch of $\overline{g}$ with respect to $g$ and its Liouville measure.
\end{itemize}
 \end{thm}
 
 The geodesic stretch was introduced and called intersection by Croke and Fathi  in \cite{croke-fathi}.

\subsection{Applications} 

We describe some consequences of Theorem \ref{intersection.number.thm}.  We say $\Pi\in S_\eta(M)$ if the asymptotic boundary of the lift to $\H^3$ of a surface in $\Pi$ is a $(1+\eta)$-quasi-circle.\medskip

\noindent{\bf Minimal surface entropy.} In \cite{calegari-marques-neves}, Calegari and the authors defined  the {\em minimal surface entropy} $E(g)$ of $(M,g)$ by
\begin{equation}\label{asymp.area}
E(g)=\lim_{\eta\to 0}\liminf_{L\to\infty}4\pi \frac{\ln \#\{\Pi\in {\mathcal S}_{\eta}(M): \text{area}_g(\Pi)\leq L\}}{L\ln L}.
\end{equation} 
In Proposition \ref{alternative.definition} it is proved that  formula (\ref{asymp.area}) is also true by taking the limsup.   It is shown in  \cite{calegari-marques-neves} that $E(\overline{g})=2$ and that, for a metric $g$ with sectional curvature less than or equal to $-1$, $E(g)\geq E(\overline{g})=2$ and equality implies $g$ is hyperbolic.  In \cite{lowe-neves}, it is shown that $E(g)\leq 2$ for any metric $g$ with scalar curvature greater than or equal to $-6$, and equality implies $g$ is a hyperbolic metric. 

Let $\mu$ be a unit Radon measure on  the Grassmannian $Gr_2(M)$ of $M$ that is invariant by the homogeneous action of ${\rm PSL}(2,\mathbb{R})$. One such example is the normalized  Liouville measure  of the metric $\overline{g}$  on $Gr_2(M)$ which we denote by $\overline{\nu}$. The theory of   Ratner-Shah \cite{ratner,shah} classifies those  measures as  linear combinations of $\overline{\nu}$ and measures supported on closed totally geodesic surfaces.

Let $\rho$ be a metric  which topologizes   the space of unit Radon measures on $Gr_2(M)$. There is $\tilde{\eta}>0$ such that if  $\Pi\in S_\eta(M)$, $0<\eta<\tilde{\eta}$, then there is a unique $\Sigma\in \Pi$ such that ${\rm area}_{\overline{g}}(\Sigma)={\rm area}_{\overline{g}}(\Pi)$ (combining \cite{seppi} and \cite[Theorem 3.3]{uhlenbeck}). We set $\Sigma=\Sigma(\Pi)$ and define a unit measure on $Gr_2(M)$ 
$$
\nu_\Pi(\psi)=\frac{1}{\text{area}_{\overline{g}}(\Sigma(\Pi))}\int_{\Sigma(\Pi)}\psi(x,T_x\Sigma(\Pi))\, d\Sigma(\Pi)_{\overline{g}}(x)
$$
 where $\psi:Gr_2(M)\rightarrow\mathbb{R}$ is a continuous function and $T_x\Sigma$ denotes the tangent space to $\Sigma$ at $x$. 

We  define the {\em minimal surface entropy $E_\mu(g)$ for $\mu$}  of $(M,g)$  as being
\begin{equation}\label{entropy.measure}
E_\mu(g)=\lim_{\eta\to 0}\liminf_{L\to\infty}4\pi \frac{\ln \#\{\Pi\in {\mathcal S}_{\eta,\mu}(M): \text{area}_g(\Pi)\leq L\}}{L\ln L},
\end{equation} 
where  $\mathcal{S}_{\eta,\mu}(M)=\{\Pi\in S_\eta(M): \rho(\nu_\Pi,\mu)<\eta\}$. Since the space of unit Radon measures on $Gr_2(M)$ is compact, the definition of $E_\mu(g)$ does not depend on the metric $\rho$. Notice that $E_\mu(g)\leq E(g)$ for any $\mu$.

 We use Theorem \ref{intersection.number.thm} to prove a sharp bound between $E_{\overline{\nu}}$ and $h_L(g)$, the Liouville entropy of the geodesic flow of $g$.

\begin{thm}\label{product.entropy.thm}  Suppose that $g$ is a metric of negative sectional curvature on $M$.  Then 
$$E_{\overline{\nu}}(g)h_L(g)\text{\rm vol}_g(M)\leq 4\text{\rm vol}_{\overline{g}}(M),$$
and if equality holds the metric $g$ has constant sectional curvature.
\end{thm}

We briefly sketch the proof of the inequality in Theorem \ref{product.entropy.thm}. We first show the existence  of $\Pi_i\in \mathcal S_{1/i}(M)$ so that
$$E_{\overline{\nu}}(g)=2\lim_{i\to\infty}\frac{{\rm area}_{\overline{g}}(\Pi_i)}{{\rm area}_{g}(\Pi_i)}\quad\mbox{and}\quad\lim_{i\to\infty}\frac{1}{{\rm area}_{\overline{g}}(\Pi_i)}\delta_{\Pi_i}=\frac{1}{2\pi{\rm vol}_{\bar g}(M)}\,{\nu}.$$
Knieper \cite[Theorem 1.1]{knieper-stretch} showed that $i(g,\overline{g})\geq \frac{h_L(g)}{2}$. The form $I$ is continuous  at $(\nu,\lambda)$ for any $\lambda\in  \mathcal C(\Gamma)$  by Proposition \ref{continuity.form} (also \cite{brody-reyes}). Hence, by Theorem \ref{intersection.number.thm}, 
\begin{multline}\label{sktech.proof}
\text{\rm vol}_g(M)h_L(g)\leq 2 \text{\rm vol}_g(M)i(g,\overline{g})=\frac{ I(\nu,\lambda_g)}{\pi^2}
\\=2\frac{{\rm vol}_{\bar g}(M)}{\pi}\lim_{i\to\infty}I(\frac{\delta_{\Pi_i}}{{\rm area}_{\overline{g}}(\Pi_i)},\lambda_g)
\leq 2{\rm vol}_{\bar g}(M) \lim_{i\to\infty}\frac{{\rm area}_{{g}}(\Pi_i)}{{\rm area}_{\overline{g}}(\Pi_i)}\\=4\frac{{\rm vol}_{\bar g}(M)}{E_{\overline{\nu}}(g)}.
\end{multline}

The   entropy $E_{\overline{\nu}}(g)$  counts only the homotopy classes for which the area-minimizing  surfaces in the hyperbolic metric equidistribute in the Grassmannian $Gr_2(M)$ of $M$.   Such sequences of homotopy classes are constructed in   \cite[Proposition 6.1]{lowe-neves} using the work of Kahn-Markovic \cite{kahn-markovic2}  (see also \cite[Theorem 5.7]{labourie}).  In \cite{kahn-markovic-smilga} there is a description of general limiting measures.

If the sectional curvatures of $g$ are less than or equal to $-1$, the same proof as in \cite{calegari-marques-neves} implies that $E_{\overline{\nu}}(g)\geq 2$
and equality holds if and only if $g$ is hyperbolic. 
Since in this case, by Pesin's formula (\cite{pesin}),  $h_L(g)\geq 2$ it follows from Theorem \ref{product.entropy.thm} that 
${\rm sec}_g\leq -1$ implies ${\rm vol}_g(M)\leq {\rm vol}_{\overline{g}}(M)$. This is known and was proven in the general $n$-dimensional case in \cite{BCG-applications}.


The previous theorem implies the following corollary for the minimal surface entropy.

\begin{cor}\label{totally.geodesic.entropy} We have
 $$E(g)h_L(g)\text{\rm vol}_g(M)\leq 4\text{\rm vol}_{\overline{g}}(M)$$
 for any metric of negative sectional curvature $g$  on $M$
if and only if   $(M,\overline{g})$ has no closed totally geodesic surfaces.  In this case, if equality holds then the metric $g$ has constant sectional curvature.
\end{cor}

Indeed, if $(M,\overline{g})$ does not have compact totally geodesic surfaces, then $E_{\overline{\nu}}(g)=E(g)$ (Proposition \ref{no.totally.geodesic}) and so Theorem \ref{product.entropy.thm} applies. If $(M,\overline{g})$ admits a closed totally geodesic surface we prove in Proposition \ref{not.true} that the inequality in Corollary \ref{totally.geodesic.entropy} fails for some perturbation of $\bar g$.

There are closed hyperbolic three-manifolds which do not have compact totally geodesic surfaces (\cite{MacR}). \medskip

\noindent{\bf Area ratio.}  
In \cite[Section 2.4.C]{gromov2},  Gromov defined 
$$A(g/\bar g)=\frac{1}{2\pi\, {\rm vol}_{\overline{g}}(M)}\int_{Gr_2(M)} \frac{{\rm area}_g}{{\rm area}_{\overline{g}}}(x,P)dV_{\overline{g}},$$
where for $(x,P)\in Gr_2(M)$, 
$\frac{{\rm area}_g}{{\rm area}_{\overline{g}}}(x,P)=\lim_{\delta\to 0}\frac{{\rm area}_g(P_{\delta})}{\delta},$
with $P_{\delta}\subset M$  a surface tangent to $P$ at $x\in M$ and  $\bar g$-area $\delta$, $dV_{\overline{g}}$ the Liouville measure.

For surfaces the analogue quantity to the area ratio was introduced by Katok \cite{katok}. In \cite[Corollary 1.3]{lowe-neves} it is shown that for any metric $g$ with scalar curvature greater than or equal to $-6$ we have ${\rm Area}(g/\overline g)\geq 1$ and equality occurs if and only if $g=\bar g$.

In Theorem 1.2 of \cite{lowe-neves} it was shown $A(g/\bar g)E(g)\geq 2$ for every metric $g$ and that equality implies that $g$ is a multiple of $\bar g$.  In Section \ref{proof.second.thm} we show that $A(g/\bar g)E_{\overline{\nu}}(g)\geq 2$ for any metric $g$. Combining with Theorem \ref{product.entropy.thm} we deduce:

\begin{cor}\label{area.distortion} Let $g$ be a metric of negative sectional curvature on $M$. Then
$$
 h_L(g) {\rm vol}_g(M)\leq 2 A(g/\bar g)\text{\rm vol}_{\overline{g}}(M),
$$
and if equality holds the metric $g$ is a constant multiple of $\bar g$.
\end{cor}

\noindent{\bf Mostow Rigidity Theorem.} The Mostow Rigidity Theorem states that  closed hyperbolic manifolds of dimension greater than or equal to three that are homotopically equivalent are isometric. There are different proofs for this theorem (e.g. \cite{BCG}, \cite{gromov-mostow},\cite{mostow}). In Section  \ref{mostow.rigidity} we give a new proof for the three-dimensional case.

 Our proof combines some generalizations of Theorem \ref{intersection.number.thm} for maps and minimal surface theory. We sketch the argument in the following particular case. Suppose that  $N=\H^3/\Gamma_N$ is diffeomorphic to $M=\H^3/\Gamma_M$ by a map  $E:M\rightarrow N$.  We want to show that $E$ is homotopic to an isometry. 
 
 The map $E$ induces a  homeomorphism on the boundary at infinity $\tilde{E}:S^2\rightarrow S^2$. The strategy consists to show that $\tilde E$ maps round circles to round circles. If that is the case, a  theorem of Carath\'eodory  \cite{caratheodory} implies the map $\tilde{E}$ is a M\"obius transformation and so there is a unique $I\in \text{Isom}(\H^3)$ so that $I_{|S^2}=\tilde{E}$. Then standard arguments  show that $I$ induces an isometry between $M$ and $N$.
 
Set $g=E^*(\bar g_N)$  and without loss of generality assume that $\text{\rm vol}_{\overline{g}}(M)\leq \text{\rm vol}_g(M)$. Consider the sequence $\Pi_i\in S_{1/i}(M)$ used in deducing  \eqref{sktech.proof} and  let $S_i\in \Pi_i$ denote the minimal surface with respect to $g$ for which ${\rm area}_{{g}}(\Pi_i)={\rm area}_{{g}}(S_i)$.   From the volume inequality  and  $h_L(g)=2$, we see from \eqref{sktech.proof} that $1\leq \lim_{i\to\infty}\frac{{\rm area}_{{g}}(\Pi_i)}{{\rm area}_{\overline{g}}(\Pi_i)}$. We then use the Gauss-Bonnet Theorem to show that $\frac{1}{\text{area}_{g}(S_i)}\int_{S_i}{|A_{S_i}|^2}dA_{g}\to 0$, where $A_{S_i}$ denotes the second fundamental form of $S_i$ with respect to $g$.

In summary, we know that the area minimizer in $\Pi_i$ for the metric $\bar g$ is becoming totally geodesic as $i\to\infty$ (from \cite{seppi,uhlenbeck}) and we have argued that the  area minimizer in $\Pi_i$ for the metric $E^*(\bar g_N)$ is also becoming totally geodesic as $i\to\infty$. In Section \ref{mostow.rigidity} we show that suffices to see that $\tilde E$ maps round circles to round circles.

 \section{Preliminaries}\label{definition}

   Let $\H^3$ be the  three-dimensional hyperbolic space. We use the  Poincar\'e unit ball model  $B^3\subset \R^3$.  The  compactification $\overline \H^3$ of $\H^3$  becomes identified with  $\overline B^3$, and the boundary at infinity  $\partial_{\infty}\H^3$  with $\partial \overline B^3=S^2$. If $X\subset \H^3$, we denote by  $\overline X\subset \overline \H^3$ its closure in the cone topology (the Euclidean topology on $\overline{B}^3$) and  set 
   $\partial_{\infty}X=\overline X\cap \partial_{\infty}\H^3=\overline X\cap S^2.$ The group $\text{Isom}_+(\H^3)$ of    orientation-preserving isometries  of the three-dimensional hyperbolic space can be identified by restriction to the boundary with the group of orientation-preserving conformal automorphisms of $S^2$, denoted by $\text{M\"{o}b}(S^2)$.

Let  $ \mathcal{G}$  be  $(S^2\times S^2-\{(x,x):x\in S^2\})/\sim$, where $(x,y)\sim(y,x)$,  with the quotient topology. The group $\text{M\"{o}b}(S^2)$  acts diagonally on $\mathcal G$.  We denote by $G(\H^3)$ the space of unoriented complete geodesics in $\H^3$  with the compact-open topology. If  $\gamma\in G(\H^3)$, then $\partial_{\infty}\gamma$ consists of two  points of $S^2$.  
 The  map $\partial_{\infty}:G(\H^3)\rightarrow \mathcal G$  is a homeomorphism. 
 
Consider  $\Gamma\subset  \text{Isom}_+(\H^3)$  a discrete group of isometries   such that $M=\H^3/\Gamma$ is a closed  manifold. Let $\pi:\H^3 \rightarrow \H^3/\Gamma$ be the canonical projection. 
 A  Riemannian metric $g$ on $M$ can be pulled back to  a metric on $\H^3$ denoted also  by $g$. We will  assume that $g$ has negative sectional curvature and we denote by $\overline{g}$ the hyperbolic metric.
 
 A geodesic $\gamma$ in  the metric $g$ in $\H^3$ defines  a unique hyperbolic geodesic  $\tilde{\gamma}$ such that the Hausdorff distance between  $\gamma$ and $\tilde{\gamma}$ is bounded. The geodesic $\gamma$ is determined by $\partial_{\infty}\gamma$. Hence the boundary at infinity of $(\H^3,g)$ is canonically identified with $\partial_\infty{\H^3}=S^2$. With $G(\H^3,g)$ denoting the set of all unoriented complete geodesics in $(\H^3,g)$ with the compact-open topology, the map $\partial_{\infty}: G(\H^3,g)\rightarrow \mathcal G$ is a  homeomorphism and we denote the inverse map  by $u\in\mathcal G\mapsto \gamma_g(u)\in G(\H^3,g)$.

The conjugacy class of $\alpha\in \Gamma$ and the free homotopy class of of a curve $\sigma$ are  denoted by $[\alpha]$ and $[\sigma]$ respectively. These classes are in one-to-one correspondence with each other. Every  homotopy class $[\sigma]$ contains a unique closed geodesic $\gamma$ with respect to $g$ which is also the curve of least length in $[\sigma]$. We use $l_g([\gamma])$ or $l_g(\gamma)$ denote the length of $\gamma$ with respect to $g$.

Let  $T^1_g\H^3$ and   $T^1_gM$ denote the unit tangent bundles of $(\H^3,g)$ and $(M,g)$. The  Liouville measure on either $T^1_g\H^3$ or   $T^1_gM$ is denoted by $dV_g$ and $dV_g(T^1_gM)=4\pi\text{\rm vol}_g(M)$.  If we  do not specify the Riemannian metric,  it is because we are referring  to the hyperbolic metric $\overline{g}$ on either $M$ or $\H^3$.

\subsection{Quasi-circles}\label{quasicircles}

A simple closed curve $\sigma\subset S^2$ is called a {\em $(1+\eta)$-quasi-circle} if there is a $(1+\eta)$-quasi-conformal homeomorphism $f:S^2\rightarrow S^2$ 
such that  $f(\{z\in \mathbb{C}:|z|=1\})=\sigma$.  Let   ${QC}_\eta$ be the space of  $(1+\eta)$-quasi-circles of $S^2$  with the topology induced by the Hausdorff distance $d_{\mathcal{H},S^2}$. Hence ${QC}_0$ is the space of geodesic circles in $S^2$, which is  homeomorphic to a punctured $\RP^3$.
 The group  $\text{M\"{o}b}(S^2)$ acts continuously on ${QC}_\eta$ by taking the image $\sigma \mapsto \phi(\sigma)$, $\phi \in \text{M\"{o}b}(S^2)$.
 
If a sequence of $(1+\eta)$-quasi-circles $\{\sigma_n\}$ is such that for different points $x,y,z\in S^2$, there exists $\{x_n,y_n,z_n\}\subset \sigma_n$  with $x_n\rightarrow x$, $y_n\rightarrow y$, $z_n\rightarrow z$, then there is a subsequence $\{\sigma_j\}\subset \{\sigma_n\}$ that converges to a $(1+\eta)$-quasi-circle $\sigma$ containing $x,y,z$.  Hence, for any $r>0$  the set $K_r\subset{QC}_\eta$ of quasi-circles  which are not contained in some geodesic ball $B_r(x)\subset S^2$ is compact. The space $QC_\eta$ is then locally compact and Hausdorff. If $\eta<\eta'$, then $QC_\eta$ is a closed subset of $QC_{\eta'}$. Also $QC_\eta$ is a $\sigma$-compact topological space, since $QC_\eta = \cup_{n \in \mathbb{N}} K_{1/n}$.

\subsection{Convex hull}\label{convex.hull}

A set $C\subset \overline \H^3$ is convex if the geodesic connecting any pair of points in $C$ is contained in $C$. For  a closed set $X\subset \partial_{\infty}\H^3$, its {\em convex hull} $\text{conv}(X)\subset \overline \H^3$   is the intersection of all  convex  subsets of   $\overline \H^3$ which contain $X$.  If $\sigma\in QC_0$, then $\text{conv}(\sigma)$ is a totally geodesic disk $D(\sigma)$ with $\partial_{\infty}D(\sigma)=\sigma$. 

 The convex hull  satisfies some continuity properties. Consider a  sequence $(\sigma_i,\partial_{\infty}\gamma_i)$ in $QC_\eta\times \mathcal G$ converging to $(\sigma,\partial_{\infty}\gamma)$, where $\gamma_i, \gamma\in G(\H^3)$, $i\in \N$.  Then \cite[Theorem 1.4]{bowditch} implies that $\text{conv}(\sigma_i)$ and $\bar \gamma_i$  converge, respectively,  in the Hausdorff distance of  the Euclidean metric on $\overline B^3$, to $\text{conv}(\sigma)$ and $\bar \gamma$.

\subsection{Integral geometry}\label{integral.geometry} 

The Grassmannian of unoriented $2$-planes in $\H^3$ is denoted by $Gr_2(\H^3)$. Given $(x,P)\in Gr_2(\H^3)$, there is  a unique totally geodesic disk $D(x,P)\subset \H^3$ with $x\in D(x,P)$ and $T_xD(x,P)=P$. Thus we get a continuous map $\theta:Gr_2(\H^3)\rightarrow QC_0$ defined by
\begin{equation}\label{theta}
 \theta(x,P)=\partial_{\infty}D(x,P).
\end{equation}

Let $dV$ be the natural volume measure on $Gr_2(\H^3)$. Santal\'o's formula (see Chapter 14 of \cite{santalo}) gives a Radon measure  $\nu$ on $QC_0$ that is invariant  under $\text{M\"{o}b}(S^2)$, and such that for every region $\Omega\subset \H^3$, every surface $S\subset \H^3$, and every curve $c\subset \H^3$ we have, respectively,
\begin{equation}\label{santalo1}
\int_{QC_0} \text{area}(D(\sigma)\cap \Omega)\,d\nu(\sigma)=2\pi\, \text{\rm vol}(\Omega),
\end{equation}
\begin{equation}
\int_{QC_0} \text{length}(D(\sigma)\cap S)\,d\nu(\sigma)=\frac{\pi^2}{2} \, \text{area}(S),
\end{equation}
and 
 \begin{equation}\label{santalo2}
\int_{QC_0} \#(D(\sigma)\cap c)\, d\nu(\sigma)=\pi\, \text{length}(c),
\end{equation}
 If   $\psi \in C_c(Gr_2(\H^3))$, then 
\begin{equation}\label{coarea}
\int_{Gr_2(\H^3)}\psi\, dV =\int_{QC_0}\int_{D(\sigma)}\psi(x,T_xD(\sigma)) \, dA(x) \,d\nu(\sigma).
\end{equation}
The  measure $dV$  also induces a volume measure $dV$ on $Gr_2(M)$, the Grassmanian of unoriented $2$-planes of $M$, whose total volume is $2\pi\text{\rm vol}(M)$.

For a negatively curved metric $g$ on $M$, we  consider the maps
\begin{equation}\label{theta1}
p_1:T^1_g(\H^3)\rightarrow \H^3, \quad p_1(x,v)=x
\end{equation}
and
\begin{equation}\label{theta1}
p_2:T^1_g(\H^3)\rightarrow \mathcal G,\quad p_2(x,v)=\partial_{\infty}\gamma_g(x,v),
\end{equation}
where $\gamma_g(x,v)$ is the unique geodesic in $(\H^3,g)$ determined by $(x,v)\in T^1_g(\H^3)$. For each $x\in\H^3$,  the metric $g$ induces a natural area measure $\omega_x$ on the fiber $p_1^{-1}(x)$ with area $4\pi$.

 In \cite[IV.19.4]{santalo},  a measure $\lambda_g$ on $\mathcal G$ (called the {\em Liouville measure on $\mathcal G$}) is defined so that for every surface $\Sigma\subset \H^3$ and every nonnegative measurable function $F$ in $\mathcal G$ we have 
\begin{eqnarray}\label{santalo.formula0}
&&\nonumber \int_{\mathcal G}\#(\gamma_g(u)\cap \Sigma) F(u)d\lambda_g(u)\\
&&\hspace{2cm} =\frac{1}{2}\int_\Sigma \int_{p_1^{-1}(x)}F\circ p_2(x,v) | \langle v, N_\Sigma(x)\rangle |d\omega_x(v) dA_g(x),
\end{eqnarray}
where $N_\Sigma$ is a unit normal to $\Sigma$ for the metric $g$ and $\gamma_g(u)\in G(\H^3,g)$ is the unique geodesic such that $\partial_{\infty}\gamma_g(u)=u$. In particular,
\begin{equation}\label{santalo.formula}
\int_{\mathcal G}\#(\gamma_g(u)\cap \Sigma)d\lambda_g(u)=\pi\text{area}_g(\Sigma).
\end{equation}
Also, if  $\psi \in C_c(Gr_1(\H^3))$ then 
\begin{equation}\label{measure.on.geodesics}
\int_{Gr_1(\H^3)}\psi\, dV_g =\int_{\mathcal{G}}\int_{\gamma_g(\sigma)}\psi(x,T_x\gamma_g(\sigma)) \, ds_g(x) \,d\lambda_g(\sigma),
\end{equation}
where $ds_g$ is arc length in the metric $g$ and  $dV_g$ is the natural volume measure on $Gr_1(\mathbb{H}^3)$ induced by $g$.

\subsection{Geodesic stretch} Let $(N,g_1)$, $(M,g_2)$ be compact manifolds with   negative sectional curvature and  $F:N\rightarrow M$ be a continuous  map.  We denote by $\pi(v)$ the basepoint of a vector $v$.   The geodesic flow with respect to the metric $g_1$ is denoted by 
$g_1^t:T^1_{g_1}(N)\rightarrow T^1_{g_1}(N)$,  or $g_1^t:T^1_{g_1}(\H^3)\rightarrow T^1_{g_1}(\H^3)$, $t\in\R$. 
For  $v\in T^1_{g_1}(N)$, denote by $a(t,v)$ the minimum length in  the $g_2$-metric of paths in $M$ homotopic with fixed endpoints to $$s\mapsto F(\pi(g_1^s(v))), \quad 0\leq s\leq t.$$ If $\gamma$ is a unit speed geodesic in $(\H^3,g_1)$ so that $(\gamma(0),\gamma'(0))$ is mapped to $v\in T^1_{g_1}(N)$, and if $F:\H^3\rightarrow \H^3$ denotes a lift of $F:N\rightarrow M$, 
then $a(v,t)=\text{d}_{g_2}(F(\gamma(0)),F(\gamma(t)))$.

  In \cite{croke-fathi}, Croke and Fathi defined the {\em geodesic stretch}  of $g_2$ with respect to $g_1$ under the map $F$ as
\begin{equation}
i_F(g_1,g_2)=\frac{1}{ \text{\rm vol}_{g_1}(T^1_{g_1}N)}\lim_{t\to\infty}\int_{T^1_{g_1}(N)}\frac{a(t,v)}{t}dV_{g_1}(v).
\end{equation}
If $F$ is the identity map, we denote $i_{id}(g_1,g_2)$ by  $i(g_1,g_2)$. The geodesic stretch is referred to as the intersection of $g_1,g_2$ in \cite{croke-fathi}.

\subsection{Essential surfaces}\label{essential.section}

A {\em surface group} is a subgroup $G<\Gamma$  such that  there is a closed surface $\Sigma$ and an immersion $i:\Sigma\rightarrow M$ that is $\pi_1$-injective with $i_*(\pi_1(\Sigma))=G$. Such immersions are called {\it essential surfaces}. We denote by $\Pi$ the conjugacy class of $G$ in $\Gamma$. The set of conjugacy classes of surface groups is in one-to-one correspondence with the set of free homotopy classes of essential surfaces. 
We denote $\Sigma \in \Pi$ if the associated immersion $i:\Sigma \rightarrow M$ satisfies  $i_*(\pi_1(\Sigma))\in \Pi$. Hence we refer to  $\Pi$ as a  {\em homotopy class of essential surfaces.}

 The immersion $i:\Sigma\rightarrow M$ lifts to a $G$-invariant immersion $\tilde i:D\rightarrow \H^3$, where $D$ is the universal cover of $\Sigma$.  Hence 
 $\tilde{i}(\phi(x))=i_*(\phi)(\tilde{i}(x))$ for every $x\in D$ and $\phi\in \pi_1(\Sigma)$. 
 We will  identify $\Sigma$ and  $D$ with their images in $M$ and $\H^3$ under the immersions $i$ and $\tilde i$, respectively, depending on the context. 
 If $\Lambda(G)\subset S^2$ denotes the limit set of $G$, defined as the set of accumulation points in  $\overline{\mathbb{H}}^3$ of an orbit $Gx$, for $x\in \mathbb{H}^3$, then $\partial_{\infty}D=\Lambda(G)$. The group $G$ preserves $\Lambda(G)$, and if $H\in \Pi$ is conjugate to $G$ there is $\phi\in \Gamma$ so that $\phi(\Lambda(H))=\Lambda(G)$. The group $G$ is called quasi-Fuchsian if the limit set $\Lambda(G)$ is a Jordan curve, in which case it is   a quasi-circle (see \cite[Prop 8.7.2]{thurston}). We define $\mathcal S_\eta(M)$ to be the set of all conjugacy classes of quasi-Fuchsian surface groups with limit set  in $QC_\eta$.

If $g$ is a Riemannian metric on $M$, we denote the area of $\Sigma\in \Pi$ with respect to $i^*g$ by $\text{area}_g(\Sigma)$.   For  a homotopy class  $\Pi$, we define
\begin{equation}\label{minimal.area.defi}
\text{area}_{g}(\Pi)= {\text{inf}} \{\text{area}_{g}(\Sigma): \Sigma\in \Pi\}.
\end{equation}
There exists  an immersed minimal surface (with respect to $g$) $\tilde{\Sigma} \in\Pi$ with
$\text{area}_g(\tilde{\Sigma})=\text{area}_{g}(\Pi)$ (e.g. \cite{schoen-yau}).  We set $\text{area}(\Pi)=\text{area}_{\overline{g}}(\Pi)$. 
 
Each $\sigma\in QC_{0}$ has a totally geodesic disk as the unique solution to the asymptotic Plateau problem.  A standard compactness argument shows the existence of $\tilde{\eta}>0$ such that for any $\sigma\in QC_{\tilde{\eta}}$ the area minimizing surface $D\subset \H^3$ with $\partial_{\infty}D=\sigma$ produced in \cite[Theorem 2.1]{anderson} has principal curvatures bounded from above  by a constant  $\lambda <1$. Uhlenbeck \cite{uhlenbeck} uniqueness argument can be applied in this setting (\cite[Lemma 2.2]{jiang}) to show that $D$ is a disk, unique, and we set $D(\sigma)=D$ (see also \cite{huang-lowe-seppi} for other properties). Hence any $\Pi\in \mathcal{S}_{\tilde{\eta}}(M)$ has  a unique minimal  surface $\Sigma\in \Pi$ which we denote by $\Sigma(\Pi)$  (\cite{uhlenbeck}).

{\subsection{Minimal surface entropy} In \cite{calegari-marques-neves}, Calegari and the authors defined  the {\em minimal surface entropy} $E(g)$ of $(M,g)$. This is the invariant
\begin{equation}
E(g)=\lim_{\eta\to 0}\liminf_{L\to\infty}4\pi \frac{\ln \#\{\Pi\in {\mathcal S}_{\eta}(M): \text{area}_g(\Pi)\leq L\}}{L\ln L}.
\end{equation}

Let 
\begin{equation}\label{quantitya}
A(g,\eta)=\inf\{\text{area}_g(\Pi)/\text{area}(\Pi):\Pi\in\mathcal S_{\eta}(M)\},
\end{equation}
and 
$A(g)=\lim_{\eta\to 0}A(g,\eta).$

\begin{prop}\label{alternative.definition}The minimal surface entropy satisfies
\begin{eqnarray*}
E(g)&=&\lim_{\eta\to 0}\limsup_{L\to\infty}4\pi \frac{\ln \#\{\Pi\in {\mathcal S}_{\eta}(M): \text{\rm area}_g(\Pi)\leq L\}}{L\ln L}\\
&=&\frac{2}{A(g)}.
\end{eqnarray*}
 \end{prop}
\begin{proof}
Let $\Pi\in S_\eta(M)$. If $\Sigma \in \Pi$ is a minimal surface in the hyperbolic metric $\overline{g}$, then by the Gauss-Bonnet theorem
$$
\text{area}(\Sigma)=4\pi(g(\Pi)-1)-\frac12 \int_\Sigma |A_\Sigma|^2 dA_{\overline{g}},
$$
where $g(\Pi)$ is the genus of $\Sigma$ and $A_\Sigma$ is its second fundamental form. By \cite{seppi}, $|A_\Sigma|\leq C \ln(1+\eta)$. Hence
$$
\lim_{\eta\rightarrow 0} \frac{\text{area}(\Pi)}{4\pi(g(\Pi)-1)}=1.
$$

 If $\text{area}_g(\Pi)\leq L$, then we have from \eqref{quantitya} that $\text{area}(\Pi)A(g,\eta)\leq L$.  Hence, for any $\alpha>0$, there exists $\eta'>0$ such that if $0<\eta<\eta'$ then
$$
4\pi(g(\Pi)-1)\leq (1+\alpha)\text{area}(\Pi)\leq (1+\alpha)A(g,\eta)^{-1}L.
$$
In particular, $g(\Pi)\leq h= (4\pi)^{-1}(1+\alpha)A(g,\eta)^{-1}L+1$. Therefore
\begin{eqnarray*}
\#\{\Pi\in \mathcal S_{\eta}(M): \text{area}_g(\Pi)\leq L\}
\leq \#\{\Pi\in \mathcal S_{\eta}(M) : g(\Pi)\leq h\}
\leq (c_2 h)^{2h},
\end{eqnarray*}
where we are using the work of Kahn and Markovic \cite{kahn-markovic} in the last inequality. Since
\begin{multline*}
4\pi \frac{\ln \#\{\Pi\in {\mathcal S}_{\eta}(M): \text{area}_g(\Pi)\leq L\}}{L\ln L}\leq 4\pi \frac{2h\ln(h)+2h\ln c_2}{L\ln L}\\
\leq 4\pi \frac{2(L^{-1}h) L\ln(L)+2(L^{-1}h)\ln(L^{-1}h)L+(2L^{-1}h\ln c_2) L}{L\ln L},
\end{multline*}
 it follows that
$$
\limsup_{L\to\infty}4\pi \frac{\ln \#\{\Pi\in {\mathcal S}_{\eta}(M): \text{area}_g(\Pi)\leq L\}}{L\ln L}\leq 8\pi(4\pi)^{-1}(1+\alpha)A(g,\eta)^{-1}.
$$
Therefore, since $\alpha>0$ is arbitrary,
\begin{equation}\label{limsup.entropy}
\lim_{\eta\to 0}\limsup_{L\to\infty}4\pi \frac{\ln \#\{\Pi\in {\mathcal S}_{\eta}(M): \text{area}_g(\Pi)\leq L\}}{L\ln L}\leq \frac{2}{A(g)}.
\end{equation}

Let   $\Pi_i\in\mathcal S_{1/i}(M)$ be a sequence so that  $$A(g)=\lim_{i\rightarrow \infty}\frac{\text{area}_g(\Pi_i)}{\text{area}(\Pi_i)}.$$ 
Denote by $C(i,k)$  the set of all index $k$ subgroups of $\Pi_i$. The elements of $C(i,k)$ correspond to surface subgroups obtained by taking a $k$-degree cover of an essential surface in $\Pi_i$. The essential surfaces corresponding to an element of  $C(i,k)$ have genus $g(i,k)=k(g(\Pi_i)-1)+1$.

If $\Lambda\in C(i,k)$, then $\Lambda\in \mathcal{S}_{1/i}(M)$ and  $\text{area}_g(\Lambda)\leq k\,\text{area}_g(\Pi_i)$. Hence
$$
\#\{\Pi\in \mathcal S_{1/i}(M): \text{area}_g(\Pi)\leq k\,\text{area}_g(\Pi_i)\}
 \geq \#C(i,k).
 $$
 We use the M\"uller-Puchta's formula as in the proof of \cite[Theorem 4.2]{calegari-marques-neves} to deduce  the existence of a constant $c_i>0$ so that $\#C(i,k)\geq (c_ig(i,k))^{2g(i,k)}$. 
 
 If $L\in [k\,\text{area}_g(\Pi_i), (k+1)\,\text{area}_g(\Pi_i))$, then
\begin{eqnarray*}
&&\frac{\ln \#\{\Pi\in {\mathcal S}_{1/i}(M): \text{area}_g(\Pi)\leq L\}}{L\ln L}\\
&&\geq \frac{\ln \#\{\Pi\in {\mathcal S}_{1/i}(M): \text{area}_g(\Pi)\leq k\,\text{area}_g(\Pi_i)\}}{(k+1)\,\text{area}_g(\Pi_i)\ln ((k+1)\,\text{area}_g(\Pi_i))}\\
&&\geq \frac{\ln \# C(i,k)}{(k+1)\,\text{area}_g(\Pi_i)\ln ((k+1)\,\text{area}_g(\Pi_i))}\\
&&\geq \frac{2g(i,k) \ln(c_ig(i,k))}{(k+1)\,\text{area}_g(\Pi_i)\ln ((k+1)\,\text{area}_g(\Pi_i))}\\
&&=\frac{\{2k(g(\Pi_i)-1)+2\} (\ln (k(g(\Pi_i)-1)+1)+\ln(c_i))}{(k+1)\,\text{area}_g(\Pi_i)(\ln (k+1) + \ln \text{area}_g(\Pi_i))}.
\end{eqnarray*}
This implies
 \begin{eqnarray*}
4\pi \liminf_{L\to\infty}\frac{\ln \#\{\Pi\in {\mathcal S}_{1/i}(M): \text{area}_g(\Pi)\leq L\}}{L\ln L}\geq 4\pi\frac{2(g(\Pi_i)-1)}{\text{area}_g(\Pi_i)}.
 \end{eqnarray*}
 Therefore, letting $i\rightarrow \infty$,
\begin{equation}\label{liminf.entropy}
 \lim_{\eta\to 0}\liminf_{L\to\infty}4\pi \frac{\ln \#\{\Pi\in {\mathcal S}_{\eta}(M): \text{area}_g(\Pi)\leq L\}}{L\ln L}
\geq \frac{2}{A(g)}.
 \end{equation}
 The proposition follows by combining (\ref{limsup.entropy}) and (\ref{liminf.entropy}).

\end{proof}

If $\Lambda$ is a subgroup of $\Pi$ of finite index, with $\Pi\in S_\eta(M)$, $\eta<\tilde{\eta}$, where $\tilde{\eta}$ is as in the introduction,  by uniqueness $\Sigma(\Lambda)$ is a degree $k$ cover of $\Sigma(\Pi)$ and hence
 $\nu_\Lambda=\nu_\Pi$.  This allows to generalize Proposition \ref{alternative.definition} to the entropies $E_{\mu}(g)$ with the same proof.  It follows that
 \begin{equation}\label{alternative.entropy.mu}
  E_\mu(g)=\lim_{\eta\to 0}\limsup_{L\to\infty}4\pi \frac{\ln \#\{\Pi\in {\mathcal S}_{\eta,\mu}(M): \text{\rm area}_g(\Pi)\leq L\}}{L\ln L}  
=\frac{2}{ A_\mu(g)},
 \end{equation}
 where $A_\mu(g)=\lim_{\eta\to 0}\inf\{\text{area}_g(\Pi)/\text{area}(\Pi):\Pi\in\mathcal S_{\eta,\mu}(M)\}$. As a result, $E_{\mu}(\overline g)=2$.

\section{Conformal currents and the  intersection form}\label{intersection}

Recall that $\Gamma\subset  \text{Isom}_+(\H^3)$ is  a discrete group of isometries   such that $M=\H^3/\Gamma$ is a closed  manifold. 
In \cite{bonahon}, Bonahon defined a {\em geodesic current} to be  a $\Gamma$-invariant Radon measure on $\mathcal G$.  The space of geodesic currents is denoted by $\mathcal C (\Gamma)$. Labourie \cite[Section 5.2]{labourie} defined a {\em conformal current} to be a $\Gamma$-invariant Radon measure on some $QC_\eta$.  Let us denote the space of conformal currents  by ${\mathcal{C}_2}(\Gamma)$
(see \cite{kapovich-nagnibeda} for a definition based on closed sets).

The results of this section overlap with those of Brody and Reyes \cite{brody-reyes},  which introduced the notion of   currents on systems of spheres  generalizing  conformal currents. The statements of this section and their proofs were  developed independently. 

\subsection{Intersection number}

In \cite{bonahon}, Bonahon defined the intersection number between geodesic currents. We will use an intersection number between geodesic currents and conformal currents.

Given $(\sigma,\{x,y\})\in QC_\eta\times \mathcal G$, its (unoriented) linking number $\text{lk}(\sigma,\{x,y\})$ is  $1$ if $x,y$  belong to separate connected components of $S^2-\sigma$ and  $0$ otherwise. In particular, if $\text{lk}(\sigma,\{x,y\})=1$ then $\sigma\cap \{x,y\}=\emptyset$.

Consider the open set 
\begin{equation}\label{parameter.space}
X=\{(\sigma,u)\in  QC_\eta\times \mathcal G: \text{lk}(\sigma,u)=1\}
\end{equation} 
in $QC_\eta\times \mathcal G$.
The space $X$ is  a locally compact Hausdorff space.  It is also $\sigma$-compact, since the set
$$
\{(\sigma, \{x,y\}) \in X: \sigma\in K_r, d_{S^2}(x,y)\geq r, d_{S^2}(x, \sigma)\geq r, d_{S^2}(y, \sigma)\geq r\}
$$
is compact, where $K_r$ is as in Section \ref{quasicircles}.

The group $\Gamma$ acts  on $X$ diagonally. In what follows we refer to the notation introduced in Section \ref{convex.hull}. For every $(\sigma,\partial_{\infty}\gamma)\in X$, with $\gamma\in G(\H^3)$,  the compact set $\text{conv}(\sigma)\cap \gamma$ is non-empty. This is  because in \cite[Theorem 4.1]{anderson} it is produced a minimal disk  $D\subset \H^3$ with $\partial_{\infty}D=\sigma$, which intersects $\gamma$,  and is contained in $\text{conv}(\sigma)$}.

\begin{lemm}\label{properly.discontinuous}
The group $\Gamma$ acts  freely and properly discontinuously on $X$. Hence  $X/\Gamma$ is a locally compact Hausdorff space.
\end{lemm}
\begin{proof}
Suppose $\phi\in\Gamma$ and $(\sigma,\{x,y\})\in X$ are such that $\phi(\sigma,\{x,y\})=(\sigma,\{x,y\})$, with $\phi \neq {\rm id}$. Then $\phi$ leaves  $\{x,y\}$ invariant.  If $\phi(x)=y$ then $\phi$ sends the geodesic connecting $x$ and $y$ to itself with the reverse orientation. This implies that $\phi$ has a fixed point in $\mathbb{H}^3$ which is a contradiction. Hence $\phi(x)=x$ and $\phi(y)=y$. Since  $\sigma$ does not intersect $\{x,y\}$, it follows that  $\phi^n(\sigma)$ converges to either $x$  or $y$ as $n\to\infty$. This is impossible because $\phi$ fixes $\sigma$. This proves $\Gamma$ acts freely.

To show that the action is properly discontinuous,  we prove that for every compact set $\mathcal K\subset X$, the set  $\{\phi\in \Gamma:\phi(\mathcal K)\cap \mathcal K\neq \emptyset\}$  is finite.
 Define $K=\cup_{\{(\sigma,\partial_{\infty}\gamma)\in \mathcal K\}} \{\gamma\cap \text{conv}(\sigma)\}$.
 
 \medskip
 
 {\bf Claim:}   $K\subset \H^3$ is compact.
 
 \medskip
 
 Let  $\{x_i\}\subset K$ be a sequence, with $x_i\in \gamma_i\cap \text{conv}(\sigma_i)$ for some $(\sigma_i,\partial_{\infty}\gamma_i)\in \mathcal K$.  By compactness there is $\{x_j\}\subset \{x_i\}$ with  $(\sigma_j,\partial_{\infty}\gamma_j)\to (\sigma,\partial_{\infty}\gamma)\in \mathcal K$. Theorem 1.4 in \cite{bowditch} implies that $\text{conv}(\sigma_j)\rightarrow {\rm conv}(\sigma)$ and
$\overline{\gamma}_j\rightarrow \overline{\gamma}$, both in the Hausdorff topology for the Euclidean metric. Hence
there exists $\{x_k\}\subset \{x_j\}$  that converges in the Euclidean metric to a point $p\in \overline{\gamma} \cap \text{conv}(\sigma)$.
Since $(\sigma,\partial_{\infty}\gamma)\in \mathcal K$, the set $\overline{\gamma} \cap \text{conv}(\sigma)$ is  compact  in
$\mathbb{H}^3$. Hence $x_k\rightarrow  x\in K$  in the hyperbolic metric.  This proves the claim.

\medskip

 If $\phi\in \Gamma$ and $(\sigma,\partial_{\infty}\gamma)\in \mathcal K$  are such that  
$(\phi(\sigma),\phi(\partial_{\infty}\gamma))\in \mathcal K$, then  $\phi(K)\cap K\neq \emptyset$. This is because $\phi({\rm conv}(Y))={\rm conv}(\phi(Y))$ for any $Y$. Since the group $\Gamma$ acts 
properly discontinuously on $\H^3$, it follows that  the set $\{\phi\in \Gamma: \phi(K)\cap K\neq \emptyset\}$ is finite which finishes the proof.

 \end{proof}

Let $P:X\rightarrow X/\Gamma$ denote the quotient map. Given $f\in C_c(X)$, set 
$f^\Gamma(x)=\sum_{\phi\in\Gamma}f(\phi(x))$. The function $f^\Gamma$ is well defined due to Lemma \ref{properly.discontinuous}. Since  $f^\Gamma$ is   $\Gamma$-invariant,  there is $\tilde{f}\in C_c(X/\Gamma)$ so that $f^\Gamma=\tilde{f} \circ P$.

For $(\mu,\lambda)\in {\mathcal C}_2(\Gamma)\times \mathcal C(\Gamma)$, the $\Gamma$-invariant measure $\mu\times\lambda$  induces a unique Radon measure  $[\mu\times \lambda]$ on the quotient space $X/\Gamma$ so that
$$[\mu\times\lambda](\tilde{f})=(\mu\times\lambda)(f), \quad f\in C_c(X).$$

\begin{lemm}
If  $(\mu,\lambda)\in {\mathcal C}_2(\Gamma)\times \mathcal C(\Gamma)$, then $[\mu\times\lambda](X/\Gamma)<\infty$.
\end{lemm}

\begin{proof}
For each compact set $K\subset X$, it holds that
$[\mu\times\lambda](P(K))\leq (\mu \times \lambda) (K)$. This is because
$$
(\mu \times \lambda) (K) = \inf_{f\in C_K} (\mu\times \lambda)(f),
$$
where $C_K$ is the set of functions $f \in C_c(X)$ satisfying $0\leq f\leq 1$ on $X$ and $f=1$ on $K$. We are using that $\tilde{f}=1$ on $P(K)$ for each $f\in C_K$.

Let $\Delta\subset \H^3$ be a Dirichlet fundamental domain of $M$ with closure $\bar \Delta$ and set
\begin{equation}\label{omega.definition}
Y=\{(\sigma,\partial_{\infty}\gamma)\in QC_\eta\times \mathcal G:  \text{conv}(\sigma)\cap \gamma\cap \bar \Delta\neq \emptyset\}.
\end{equation}
It follows from  Sections \ref{quasicircles} and  \ref{convex.hull}  that $Y$ is compact.
For every $z\in X$, there is $\phi\in\Gamma$ so that $\phi(z)\in Y$. Thus $P(Y \cap X)=X/\Gamma$.  

We write  $X=\bigcup_{n\in \mathbb{N}} K_n$, where $K_n \subset K_{n+1}$ and each $K_n$ is a compact set. Hence 
$Y\cap K_n \subset X$ is a compact set. Therefore 
$$
[\mu\times\lambda](P(Y\cap K_n))\leq (\mu \times \lambda) (Y \cap K_n)\leq (\mu \times \lambda) (Y)<\infty.
$$
Since $X/\Gamma = P(Y\cap X)= \bigcup_{n\in \mathbb{N}} P(Y\cap K_n)$, we get that $[\mu\times\lambda](X/\Gamma)\leq (\mu\times\lambda)(Y)$. This finishes the proof of  the lemma.
\end{proof}


 The {\em intersection form} between conformal currents and geodesic currents is the bilinear map
$$I=I_\Gamma: {\mathcal C}_2(\Gamma)\times \mathcal C(\Gamma)\rightarrow \R$$
defined by $$I(\mu,\lambda)=[\mu\times\lambda](X/\Gamma).$$

\subsection{Examples of conformal currents and geodesic currents}\label{examples.currents}

Consider  a homotopy class of quasi-Fuchsian surfaces $\Pi$ and $G$ a surface group in $\Pi$. Let $\Gamma/G=\{\phi G:\phi\in \Gamma\}$ be the set of cosets and choose a representative $\tilde \phi$ in each coset $\phi G$. Choose $\eta>0$  so that $\Lambda(G)\in QC_{\eta}$. We consider the $\Gamma$-invariant measure (see \cite[Proposition 5.5.(ii)]{labourie})
$$
\delta_{\Pi}=\sum_{\tilde{\phi}} \delta_{\tilde{\phi}(\Lambda(G))}.
$$
That $\delta_{\Pi}$ is a locally finite measure was proven in  \cite{labourie} under the condition that $\eta<\eta_1$ for some $\eta_1>0$. To see that this  holds for any $\eta$, and thus that $\delta_{\Pi} \in {\mathcal C}_2(\Gamma)$, we proceed as follows.

 Let $\delta_{\Lambda(G)}$ be the Dirac delta supported on the set  $\{{\phi(\Lambda(G))}:\phi\in \Gamma\}$ and $H=\{\phi\in \Gamma: \phi(\Lambda(G))=\Lambda(G)\}$, so $G< H$. Since the quotient of ${\rm conv}(\Lambda(G))$ by the  action
of $G$  is compact, and $H$ preserves ${\rm conv}(\Lambda(G))$, it follows that  $|H: G|<\infty$. Then $\delta_{\Pi}=|H:G|\cdot \delta_{\Lambda(G)}$.   Suppose by contradiction that $\delta_{\Lambda(G)}$ is not locally finite. 
Then one can find  $\{\phi_i\}_i\subset \Gamma$ so that $\phi_i(\Lambda(G))\rightarrow C$ in the Hausdorff distance with $\phi_i(\Lambda(G))\neq \phi_j(\Lambda(G))$ for every $i\neq j$.  Let $p,q$ be in different connected components of $S^2 \setminus C$, and $\gamma$ be the hyperbolic geodesic connecting $p$ and $q$. Since we can find a  converging sequence $\{f_i\}_i$ of  $(1+\eta)$-quasiconformal homeomorphisms of $S^2$ with $f_i(S^1)=\phi_i(\Lambda(G))$, it follows that for sufficiently large $i$, $p$ and $q$ are in separate connected components of 
$S^2 \setminus \phi_i(\Lambda(G))$.   

Let  $q_i \in {\rm conv}( \phi_i(\Lambda(G))) \cap \gamma$.
Notice that the sequence $\{q_i\}_i$ is contained in a compact set. Then $p_i'=\phi_i^{-1}(q_i)\in {\rm conv}(\Lambda(G))$. Since  ${\rm conv}(\Lambda(G))/G$ is compact, we can find $g_i\in G$ such that   $\{p_i=g_i(p_i')\}_i$ is  in a compact set.  Because the action of 
$\Gamma$ on $\mathbb{H}^3$ is properly discontinuous, there is a sequence $\{k\}\subset \{i\}$ such that there is $\psi\in \Gamma$ with  $g_k\phi_k^{-1}=\psi$ for any $k$. Hence $\phi_k=\psi^{-1}g_k$, and so $\phi_k(\Lambda(G))=\psi^{-1}(\Lambda(G))$ for any $k$. This contradicts the fact that $\phi_i(\Lambda(G))\neq \phi_j(\Lambda(G))$ for any $i\neq j$.
Hence $\delta_{\Pi}$ is locally finite.

\medskip

The Radon measure $\nu$ defined in Section \ref{integral.geometry} with support  the set $QC_0$ is $\Gamma$-invariant and thus  $\nu\in {\mathcal C}_2(\Gamma)$. \medskip

Let $[\gamma]$ be a nontrivial homotopy class of continuous maps $S^1\rightarrow M$. There is a closed geodesic $\alpha\in \Lambda$ which lifts to a geodesic $\tilde\alpha$ in $\mathbb{H}^3$ so that $\alpha=\tilde{\alpha}/H$ for some group $H<\Gamma$  isomorphic to $\mathbb{Z}$.  Let $\Gamma/H=\{\phi H:\phi\in \Gamma\}$ be the set of cosets and pick a representative $\tilde \phi$ in each coset $\phi H$.  If $\{p,q\}\in \mathcal{G}$ denote the endpoints of $\tilde{\alpha}$, then we define the $\Gamma$-invariant measure 
$$
\delta_{[\gamma]}=\sum_{\tilde{\phi}} \delta_{\tilde{\phi}(\{p,q\})}.
$$
As for $\delta_\Pi$, the measure $\delta_{[\gamma]}$ is locally finite and hence  $\delta_{[\gamma]} \in \mathcal{C}(\Gamma)$.
\medskip

Let $g$ be a metric of negative sectional curvature on   $M$. The Liouville measure $\lambda_g$ defined  in Section \ref{integral.geometry} is $\Gamma$-invariant, locally finite, and thus  $\lambda_g\in \mathcal C(\Gamma)$.

\subsection{Convergence of currents}

A sequence of homotopy classes of quasi-Fuchsian surfaces  $\{\Pi_i\}_{i\in\N}$ {\em equidistributes} if $\Pi_i\in S_{\eta_i}(M)$, $\eta_i \rightarrow 0$,   and for every $\psi\in C({Gr_2}(M))$,
\begin{equation}\label{equidis.defi}
\lim_{i\to\infty}\frac{1}{\text{area}(\Sigma(\Pi_i))}\int_{\Sigma(\Pi_i)}\psi(x,T_x\Sigma(\Pi_i))dA(x)=\frac{1}{{\rm vol}(Gr_2(M))}\int_{{Gr_2}(M)}\psi d{V}.
\end{equation}

Then:
\begin{prop}\label{convergence} Let $\{\Pi_i\}_{i\in\N}$ be a sequence of homotopy classes of quasi-Fuchsian surfaces  that equidistributes. Then 
$$\frac{1}{{\rm area}(\Sigma(\Pi_i))}{\delta_{\Pi_i}} \rightarrow \frac{1}{2\pi{\rm vol}(M)}\, {\nu}$$
as $i\to\infty$. 
\end{prop}

We will use  the following proposition.
\begin{prop}\label{locally.finite}
For every compact set $K\subset QC_{\tilde\eta}$, there is a constant $C>0$ (depending on $\Gamma$ and $K$) such that for every homotopy class of quasi-Fuchsian surfaces  $\Pi$ so that  $\sigma=\Lambda(G)\in K$ for some $G\in \Pi$,  we have
$$\#\{\phi G \in \Gamma/G: \phi (\sigma)\in K\} \leq C {\rm area}(\Pi).$$
\end{prop}
\begin{proof}
 Let $K\subset QC_{\tilde\eta}$ be a compact set, and $\Pi$ be a homotopy class of quasi-Fuchsian surfaces with $\sigma=\Lambda(G)\in K$ for some $G\in \Pi$. 
 
 \medskip

{\bf Claim:} There is a constant $R>0$ depending on $K$  so that  if $\alpha\in K$ and $D\subset \H^3$ is the minimal disk  with $\partial_{\infty}D=\alpha$, then $D\cap B_R(0)\neq \emptyset$.

\medskip

Suppose the claim is false.  Then we can find a sequence $\{\alpha_i\}_{i} \subset K$ and minimal disks $D_i$ with $\partial_{\infty}D_i=\alpha_i$ such that the sequence $\{D_i\}$ escapes any compact set of $\mathbb{H}^3$.  We can suppose  that  $\{\alpha_i\}_{i\in\N}$ converges in Hausdorff distance  to an element  $\alpha\in K$. Pick $\gamma\in G(\H^3)$ so that $\text{lk}(\alpha,\partial_{\infty}\gamma)=1$.  We know that  $\text{conv}(\alpha)\cap \gamma$ is a nonempty compact set. Hence there is $R>0$ so that $\text{conv}(\alpha)\cap \gamma\subset B_{R/2}(0)$.

As discussed in Section \ref{convex.hull}, it follows from \cite{bowditch}   that $\text{conv}(\alpha_i)$ converges to $\text{conv}(\alpha)$ in $\overline \H^3$ (with respect to the Euclidean Hausdorff distance). Thus $\text{conv}(\alpha_i)\cap \gamma\subset B_{R}(0)$ for every $i$ sufficiently large. Moreover, as in Section \ref{examples.currents},  we have  that $\text{lk}(\alpha_i,\partial_{\infty}\gamma)=1$ for every $i$ sufficiently large and so $D_i$ intersects $\gamma$ 
 for any $i$ sufficiently large.  On the other hand we know  that $D_i\subset \text{conv}(\alpha_i)$ and so $D_i\cap \gamma\subset  B_R(0)$, which contradicts the fact that $D_i\cap B_{R_i}(0)=\emptyset$  for some sequence $R_i\to\infty$. This proves the claim. 

\medskip
 
In each coset $\phi G$ \mbox{of $ \Gamma/G$} pick a representative $\tilde \phi$.  In the remainder of the proof we denote also by $\Gamma/G$ the set whose elements are these representatives.  Choose $D_{\sigma}$ a $G$-invariant minimal disk in $\H^3$ so that $\partial_\infty D_{\sigma}=\sigma$ and $D_{\sigma}/G=\tilde{\Sigma}$, where $\tilde{\Sigma}\in \Pi$ is such that ${\rm area}(\tilde{\Sigma})={\rm area}(\Pi)$. By the claim, we have 
\begin{multline}\label{loc.finite.intersection}
\#\{\tilde \phi \in \Gamma/G: \tilde \phi (\sigma)\in K\}\leq \#\{\tilde \phi \in \Gamma/G:\tilde \phi(D_{\sigma})\cap B_R(0)\neq \emptyset\}.
\end{multline}

Consider a Dirichlet fundamental domain $\Delta\subset \H^3$ for $M$ containing the origin in its interior. There is a finite set $\overline\Gamma\subset \Gamma$, depending only on $R$ and $\Gamma$, so that $B_{R}(0)\subset \cup_{\phi\in \overline\Gamma}\phi(\Delta)$. Hence if
$$\tilde \Gamma=\{\tilde\phi \in \Gamma/G: \tilde\phi(D_{\sigma})\cap\Delta\neq \emptyset\},$$
we have from \eqref{loc.finite.intersection} that
$$
\#\{\tilde \phi \in \Gamma/G: \tilde \phi (\sigma)\in K\}
\leq \#\{\tilde\phi \in \Gamma/G:\tilde\phi(D_{\sigma})\cap B_R(0)\neq \emptyset\} \leq (\#\overline{\Gamma})\cdot (\# \tilde \Gamma).
$$
 In \cite[Proposition 6.4 (ii)]{calegari-marques-neves}  it is shown the existence of $C>0$ depending on $\Gamma$ so that
$$\# \tilde \Gamma\leq  C\, \text{area}(\tilde{\Sigma}).$$
Hence $\#\{\tilde \phi \in \Gamma/G: \tilde \phi (\sigma)\in K\}\leq  C (\#\overline{\Gamma})\, \text{area}(\Pi)$, which proves the proposition.
\end{proof}

\begin{proof}[Proof of Proposition \ref{convergence}] 

For  $i\geq \tilde{i}$  the support of $\delta_{\Pi_i}$ is  in $QC_{\tilde{\eta}}$.  We define $\nu_i={\text{area}({\Sigma}(\Pi_i))}^{-1}{\delta_{\Pi_i}}.$

Let $f\in  C_c(QC_{\tilde{\eta}})$. If  $f=0$ on $QC_0$, then $\nu (f)=0$.  Let $U\subset QC_{\tilde{\eta}}$ be an open set with compact closure such that ${\rm supp}(f)\subset U$.  For any $\delta>0$  there is $0<\eta<\tilde{\eta}$ so that $\sup_{\sigma\in QC_{\eta}\cap \overline{U}}|f|(\sigma)<\delta$. Notice that $QC_{\eta}$ contains the support of $\nu_i$ for any $i$ sufficiently large.   From Proposition \ref{locally.finite}, there exists $C_1>0$ so that $|\nu_i(f)|\leq C_1\delta$ for any $i$ sufficiently large.  Hence  $\nu_i(f)\to 0$. This proves the proposition when $f=0$ on $QC_0$.

Let  $\alpha\in QC_0$,   $y\in D({\alpha})$, and  $\gamma\in G(\H^3)$ be the geodesic passing by $y$ orthogonally to $D({\alpha})$. There is an open neighborhood $U
_{\alpha}\subset QC_{\tilde{\eta}}$ of $\alpha$, with compact closure,  so that for every $\sigma\in U_{\alpha}$ the disk $D({\sigma})$ intersects $\gamma$ at a unique  point $y(\sigma)$ and  $\text{lk}(\sigma,\partial_{\infty}\gamma)=1$. The map $\sigma\in U_{\alpha}\to y(\sigma)\in \H^3$ is continuous.   We  assume its image is in a ball centered at $y$ with radius less than half the injectivity  radius of $M$. 

\begin{lemm}\label{local.convergence} For every $f\in C_c(U_{\alpha})$, we have $\nu_i(f)\to\frac{1}{2\pi{\rm vol}(M)} \nu(f)$. 
\end{lemm}
\begin{proof}

Let  $f\in C_c(U_{\alpha})$.
Recall the map $\theta:Gr_2(\H^3)\rightarrow QC_0$ defined in \eqref{theta}. With $r>0$ and $c$ constants to be chosen later, consider the following function defined on $\theta^{-1}(U_{\alpha}\cap QC_0)\subset Gr_2(\H^3):$
\begin{equation}\label{cutt-off.defi}
\psi(x,P)=c f(\sigma)(r-\dist(x,y(\sigma)))_+,\quad \sigma=\theta(x,P).
\end{equation}
Since $f$ has compact support in $U_{\alpha}$, the function $\psi$ has compact support in $\theta^{-1}(U_{\alpha}\cap QC_0)$. Thus we can extend $\psi$ so that $\psi=0$ in the complement of the set $\theta^{-1}(U_{\alpha}\cap QC_0)$. Hence  $\psi\in C_c(Gr_2(\H^3))$. 

We choose $r$ sufficiently small so that if $\psi(x,P)\neq 0$, then $x$ is in a ball centered at $y$ with radius less than the injectivity radius of $M$. Hence  there is $\tilde{\psi}\in C(Gr_2(M))$ so that if $\pi:Gr_2(\H^3)\rightarrow Gr_2(M)$ denotes the quotient map then
\begin{equation}\label{psi.quotient}
\tilde{\psi}\circ\pi=\sum_{\phi\in \Gamma}\psi\circ\phi.
\end{equation} 
Then choose $c$ so  that $2\pi c\int_0^r(r-s)\sinh sds =1$. We have  from \eqref{cutt-off.defi} that for any $\sigma\in QC_0\cap U_{\alpha}$,
\begin{equation}\label{equidis.formula}
\int_{D(\sigma)}\psi(x,T_xD(\sigma))dA=cf(\sigma)\int_{D(\sigma)}(r-\dist(x,y(\sigma))_+dA=f(\sigma).
\end{equation}
Hence  from  \eqref{coarea} it follows  that
\begin{eqnarray*}
\int_{Gr_2(\H^3)}\psi dV&=&c\int_{QC_0} f(\sigma) \int_{D(\sigma)} (r-\dist(x,y(\sigma)))_+\, dA \, d\nu(\sigma)\\
&=& \int_{QC_0} f(\sigma) \, d\nu(\sigma)=\nu(f).
\end{eqnarray*}
This implies 
$$\nu (f)=\int_{Gr_2(M)}\tilde\psi \,dV.$$

Let $V$ be an open set  such that $\text{supp}(f)\subset V\subset \bar V\subset U_{\alpha}$.  Then there is  $0<\eta_1<\tilde{\eta}$ so that if $\sigma\in QC_{\eta_1} -V$ then $\psi(x,T_xD(\sigma))=0$ for any $x\in D(\sigma)$.

Let  $h\in C_c(U_{\alpha})$ be such  that $0\leq h\leq 1$ and $h=1$ on the compact set $ \bar V$. Consider the function  $\tilde{f} \in C_c(U_{\alpha})$ defined by
$$ \tilde f(\sigma)=h(\sigma)\int_{D(\sigma)}\psi(x,T_xD(\sigma))\, dA.
$$
If $\sigma\in QC_0\cap U_{\alpha}$ then from \eqref{equidis.formula} we have
$\tilde f(\sigma)=h(\sigma)f(\sigma).$ Hence if $\sigma \in QC_0\cap V$, we have
$\tilde f(\sigma)=f(\sigma).$ If $\sigma \in QC_0\cap (U_{\alpha}- V)$, then since $\psi(x,T_xD(\sigma))=0$ for any  $x\in D(\sigma)$ it 
follows that $\tilde{f}(\sigma)=f(\sigma)=0$. 
Thus $\tilde f-f=0$ on $QC_0$ and so  $\lim_i\nu_i(\tilde f-f)=0$.

The support of $\nu_i$ is contained in $QC_{\eta_1}$ for sufficiently large $i$. We choose $G_i\in\Pi_i$, hence $\sigma_i=\Lambda(G_i)\in QC_{\eta_1}$. Let $\{\tilde{\phi}\}$ be a set of representatives of $\Gamma/G_i$. Using  that if $\sigma\in QC_{\eta_1}\cap V$ then $h(\sigma)=1$, and that if $\sigma\in QC_{\eta_1}-V$ then $\tilde f(\sigma)=0$, we deduce
\begin{align*}
  \nu_i(\tilde f) 
  &=\frac{1}{\text{area}(\Pi_i)}\sum_{\tilde\phi}\tilde f(\tilde\phi(\sigma_i))
  =\frac{1}{\text{area}(\Pi_i)}\sum_{\tilde\phi}\int_{\tilde\phi(D(\sigma_i))}\psi(x,T_x\tilde\phi(D(\sigma_i)))dA\\
 & = \frac{1}{\text{area}(\Pi_i)}\sum_{\tilde\phi}\int_{D(\sigma_i)}\psi\circ\tilde{\phi} (x,T_xD(\sigma_i))dA.
 \end{align*}
 
 
 Let $\Delta_i\subset D(\sigma_i)$ be a Dirichlet fundamental domain for the covering $D(\sigma_i)\rightarrow D(\sigma_i)/G_i=\Sigma(\Pi_i)$.  Then
\begin{eqnarray*}
 \int_{\Sigma(\Pi_i)}\tilde\psi (x,T_x\Sigma(\Pi_i))\, dA &=&\int_{\Delta_i} (\sum_{\phi\in \Gamma}\psi\circ\phi)\, dA
 =  \int_{\Delta_i}(\sum_{\tilde\phi}\sum_{g\in G_i}\psi\circ\tilde{\phi} g)\, dA\\
  &=& \sum_{\tilde\phi}\int_{\Delta_i}(\sum_{g\in G_i}\psi\circ\tilde{\phi} g)\, dA
   = \sum_{\tilde\phi}\int_{D(\sigma_i)}\psi\circ\tilde{\phi} \, dA.
\end{eqnarray*}
Therefore we have proved
$$
 \nu_i(\tilde f) = \frac{1}{{\rm area}(\Pi_i)}\int_{\Sigma(\Pi_i)}\tilde\psi (x,T_x\Sigma(\Pi_i))\, dA.
$$

Thus we have from \eqref{equidis.defi} that 
$$
\lim_{i\rightarrow \infty} \nu_i(\tilde f) =\frac{1}{{\rm vol}(Gr_2(M))}\int_{{Gr_2}(M)}\tilde\psi \,d{V}.
$$
Hence
$$
\lim_{i\rightarrow \infty} \nu_i( f)=\frac{1}{2\pi\,{\rm vol}(M)}\nu(f).
$$
\end{proof}

Let $f\in  C_c(QC_{\tilde{\eta}})$.  Consider $K=\text{supp}(f)\cap QC_0$.  By compactness  of $K$, we can find  open neighborhoods $\{U_{\alpha_j}\}_{j=1}^q$, $\{V_{\alpha_j}\}_{j=1}^q$, with $V_{\alpha_j}\subset \overline{V}_{\alpha_j}\subset U_{\alpha_j}$,   $K\subset V=\cup_{j=1}^qV_{\alpha_j}$ and such that Lemma \ref{local.convergence} holds for each $U_{\alpha_j}$.  Notice that $\overline{V}\subset U= \cup_{j=1}^qU_{\alpha_j}$. There exist nonnegative functions $\{\xi_1,\dots,\xi_q\}\subset C_c(QC_{\tilde{\eta}})$ and $\xi\in C(QC_{\tilde{\eta}})$ such that $\xi+\sum_j\xi_j=1$, 
$\overline{V}_{\alpha_j}\subset {\rm supp}(\xi_j)\subset U_{\alpha_j}$, and $(QC_{\tilde{\eta}}-V) \subset {\rm supp}(\xi) \subset (QC_{\tilde{\eta}}-K)$.

Then $f=\xi f+ \sum_{j=1}^q \xi_jf$. Since $\xi f=0$ on $QC_0$, it follows  that $\nu_i(\xi f)\rightarrow \frac{1}{2\pi\,{\rm vol}(M)}\nu(\xi f)$. By Lemma \ref{local.convergence}, 
$\nu_i(\xi_j f) \rightarrow  \frac{1}{2\pi\,{\rm vol}(M)}\nu(\xi_j f)$ for every $1\leq j\leq q$ because $\xi_j f\in C_c(U_{\alpha_j})$. Since 
$\nu_i(f)=\nu_i(\xi f)+\sum_{j=1}^q \nu_i(\xi_j f)$, the proposition follows.

\end{proof}

A sequence of closed geodesics  $\{\gamma_i\}_{i\in\N}$ {\em becomes equidistributed with respect to the metric $g$} if for any $\psi\in C(T^1_g M)$,
\begin{equation}\label{equidis}
\lim_{i\to\infty}\frac{1}{l_g(\gamma_i)}\int_{0}^{l_g(\gamma_i)}\psi(\gamma_i(t),\gamma'_i(t))\,dt=\frac{1}{{\rm vol}(T^1_gM)}\int_{T^1_gM}\psi \,dV_g.
\end{equation}

The proof of Proposition \ref{convergence} can be adapted to the case of geodesics in a metric of negative sectional curvature: 
\begin{prop}\label{convergence.geodesic}
Suppose $\{\gamma_i\}_{i\in\N}$ is a sequence of closed geodesics that is equidistributed with respect to the metric $g$.  If $\Lambda_i$ denotes the homotopy class of $\gamma_i$, then
 $$\frac{1}{l_g(\gamma_i)}\delta_{\Lambda_i}\rightarrow \frac{1}{ 2\pi{\rm vol}_g(M)}\lambda_g$$
 as $i\to\infty$.
 \end{prop}

Recall the definition of $X\subset QC_{\eta} \times \mathcal{G}$ in \eqref{parameter.space}, and that  $P:X\rightarrow X/\Gamma$ denotes  the quotient map.   

\begin{prop}\label{continuity.form}
If $\lambda\in  \mathcal C(\Gamma)$, the map $I$ is continuous  at  $(\nu,\lambda)$.
\end{prop}

\begin{proof}
 Let $\{\nu_i\}_{i\in\N}\subset  {\mathcal C}_2(\Gamma)$ and $\{\lambda_i\}_{i\in\N}\subset \mathcal C(\Gamma)$ be sequences which converge as Radon measures to $\nu$ and $\lambda$, respectively. It follows by Stone-Weierstrass  that  $[\nu_i\times \lambda_i](f)\to[\nu\times \lambda](f)$ for every $f\in C_c(X/\Gamma)$. Hence
$$[\nu\times \lambda](X/\Gamma)\leq \liminf_{i\to\infty}[\nu_i\times \lambda_i](X/\Gamma).$$

Set $\xi_i=\nu_i\times\lambda_i$ and  $\xi=\nu \times\lambda$ as Radon measures on $QC_\eta \times \mathcal{G}$.  Consider
 $$H'=\{(\sigma,\{x,y\})\in QC_\eta \times \mathcal{G}: x \in \sigma {\rm \, or \, } y \in \sigma\}$$ 
 and notice that $H'$ is invariant by the action of $\Gamma$. We will show  that $\xi(H')=0$. For each $\{x,y\}\in \mathcal{G}$, the set
$$H'_{\{x,y\}}=\{\sigma \in QC_\eta: (\sigma, \{x,y\}) \in H'\}$$ 
satisfies $\nu(H'_{\{x,y\}})=\nu(H'_{\{x,y\}} \cap QC_0)=0.$ 
This follows from (\ref{santalo2}) and the invariance of $\nu$ by  $\text{M\"{o}b}(S^2)$. Then Fubini's theorem implies that
$\xi(H')=0$.

Let $Y=\{(\sigma,\partial_{\infty}\gamma)\in QC_\eta\times \mathcal G: \gamma\in G(\H^3), \text{conv}(\sigma)\cap \gamma\cap \bar \Delta\neq \emptyset\}$ as in \eqref{omega.definition} and recall that $P(Y\cap X)=X/\Gamma$. The set $H'\cap Y$ is compact and so 
given $\delta>0$ we find  $O$  an open neighborhood  of $H'\cap Y$ with $\xi(\bar{O})\leq \delta$. 
The set $\tilde{O}=P( O\cap X)$ is open in $X/\Gamma$.   Since $[\xi_i](P(O\cap X))\leq \xi_i(O)$, it follows that $\limsup_{i\to\infty} [\xi_i](\tilde O)\leq \delta$. The closure of $X$ in $QC_\eta\times \mathcal G$ is contained in $X \cup H'$ and thus the set $K=P(Y\cap X-O)$ is compact in $X/ \Gamma$. Therefore
$$\limsup_{i\to\infty} [\xi_i](X/\Gamma)\leq \limsup_{i\to\infty} ([\xi_i](K)+ [\xi_i](\tilde O))\leq  [\xi](K)+\delta\leq [\xi](X/\Gamma)+\delta.$$
This finishes the proof.
\end{proof}

Let $[\gamma]$ be a homotopy class of closed curves.

\begin{prop}\label{length.form}
Then $I(\nu,\delta_{[\gamma]})=\pi\, l_{\overline{g}}([\gamma])$.
\end{prop}

\begin{proof}
 Let $\alpha$ be the closed geodesic in the homotopy class $[\gamma]$. Let $\tilde{\alpha}\in G(\H^3)$ be a lift of $\alpha$ with $\{p,q\}=\partial_\infty\tilde{\alpha}$ and $\alpha=\tilde{\alpha}/H$ for some $H<\Gamma$.  Consider $\{f_i\}_{i\in I}$ a
locally finite continuous partition of unity of $X/\Gamma$. Then $[\nu \times \delta_{[\gamma]}](X/\Gamma)= \sum_i [\nu \times \delta_{[\gamma]}](f_i)$.

We can suppose that the support of $f_i$ is contained in an open set $U_i\subset X/\Gamma$ such that there exists an open set $\tilde{U}_i$ with
$P:\tilde{U}_i\rightarrow U_i$ a homeomorphism. Let $\tilde{f}_i$ be the function with support contained in $\tilde{U}_i$ such that
$\tilde{f}_i=f_i \circ P$. Then $[\nu \times \delta_{[\gamma]}](f_i)=(\nu \times \delta_{[\gamma]})(\tilde{f}_i)$.

Let $J \subset \tilde{\alpha}$ be a segment of the form $[\tilde{\alpha}(t),\tilde{\alpha}(t+s))$  which is a fundamental domain for the projection $\tilde{\alpha}\rightarrow \alpha$. Let $QC_0'$ be the set of round circles $\sigma$ such that the totally geodesic disk $D_\sigma$ with $\partial D_\sigma=\sigma$
intersects $J$. Then the set of round circles $\sigma$ with ${\rm lk}(\sigma, \{p,q\})=1$ is the disjoint union $\bigcup_{h\in H} hQC_0'$.

It follows by using Fubini's theorem that
\begin{eqnarray*}
\sum_i [\nu \times \delta_{[\gamma]}](f_i)&=&\sum_i(\nu \times \delta_{[\gamma]})(\tilde{f}_i)
=\sum_i (\nu \times \sum_{\tilde{\psi}}\delta_{\tilde{\psi}(\{p,q\})})(\tilde{f}_i)\\
&=&\sum_i \int_{QC_0} \sum_{\tilde{\psi}} \tilde{f}_i(\sigma,\tilde{\psi}(\{p,q\}))\,d \nu (\sigma).\\
\end{eqnarray*}
Since  $\nu$ is a $\Gamma$-invariant measure, 
\begin{eqnarray*}
\sum_i [\nu \times \delta_{[\gamma]}](f_i)&=&\sum_i \int_{QC_0} \sum_{\tilde{\psi}} \tilde{f}_i(\tilde{\psi}\sigma,\tilde{\psi}(\{p,q\}))\,d \nu (\sigma)\\
&=&\sum_i \int_{\bigcup_{h\in H} hQC_0'} \sum_{\tilde{\psi}} \tilde{f}_i(\tilde{\psi}\sigma,\tilde{\psi}(\{p,q\}))\,d \nu (\sigma)\\
&=&\sum_i\sum_{h\in H} \int_{hQC_0'} \sum_{\tilde{\psi}} \tilde{f}_i(\tilde{\psi}\sigma,\tilde{\psi}(\{p,q\}))\,d \nu (\sigma)\\
&=&\sum_i\sum_{h\in H} \int_{QC_0'} \sum_{\tilde{\psi}} \tilde{f}_i(\tilde{\psi}h\sigma,\tilde{\psi}h(\{p,q\}))\,d \nu (\sigma)\\
&=&\sum_i\sum_{b \in \Gamma} \int_{QC_0'} \tilde{f}_i(b\sigma,\lambda(\{p,q\}))\,d \nu (\sigma).\\
\end{eqnarray*}
Hence, using \eqref{santalo2} in the last equality, 
$$
\sum_i [\nu \times \delta_{[\gamma]}](f_i)=\sum_i \int_{QC_0'} f_i(P(\sigma,\{p,q\}))\,d \nu (\sigma)
=\nu(QC_0')
=\pi l(\alpha),
$$
and the result follows because $I(\nu, \delta_{[\gamma]})=\sum_i [\nu \times \delta_{[\gamma]}](f_i)$.
\end{proof}

The intersection form is also related to the geometric intersection.

\begin{prop}\label{geometric.form}
For a homotopy class    of closed curves  $[\gamma]$ and  a homotopy class  of   quasi-Fuchsian  surfaces $\Pi$:
 \begin{itemize}
 \item[(i)] 
$I(\delta_{\Pi},\delta_{[\gamma]})\leq \#(\Sigma\cap \alpha),$
where $\Sigma$ is any immersed surface in $\Pi$ and $\alpha$ is any closed curve in $[\gamma]$ which intersect transversely; 
\item[(ii)] $I(\delta_{\Pi},\delta_{[\gamma]})= \#(\Sigma\cap \alpha),$ if $\Sigma \in \Pi$ is a  totally geodesic surface and $\alpha\in [\gamma]$ is a closed geodesic for the hyperbolic metric which intersect transversely.
\end{itemize}
\end{prop}

\begin{proof}
 Let $[\gamma]$ be a homotopy class of closed curves, and $\Pi$ be a homotopy class
of quasi-Fuchsian surfaces.  Let $\Sigma \in \Pi$ be a surface which induces the group $G < \Gamma$. Set $\sigma=\Lambda(G)\in QC_\eta$ 
and let $D\subset \H^3$ be a  $G$-invariant disk with $\partial_{\infty}D=\sigma$ and $\Sigma=D/G$. Consider  a curve $\alpha$ in $[\gamma]$ and let $\tilde{\alpha}\in G(\H^3)$ be a lift of $\alpha$ with $\{p,q\}=\partial_\infty\tilde{\alpha}$ and $\alpha=\tilde{\alpha}/H$ for some $H<\Gamma$.   If $i_\Sigma$ and $i_\alpha$
denote the immersions $i_\Sigma:\Sigma \rightarrow M$ and $i_\alpha:S^1\rightarrow M$, which intersect transversally,  the intersection number 
$\# (\Sigma \cap \alpha)$ is the number of elements in the set 
$$
Z=\{(x,y)\in \Sigma \times S^1: i_\Sigma(x)=i_\alpha(y)\}.
$$

Recall that $\delta_\Pi=\sum_{\tilde{\phi}} \delta_{\tilde{\phi}(\sigma)}$ and   $\delta_{[\gamma]}=\sum_{\tilde{\psi}}\delta_{\tilde{\psi}(\{p,q\})}$, where $\{\tilde{\phi}\}$ is a set of representatives of the cosets in $\Gamma/G$ and $\{\tilde{\psi}\}$ is a set of representatives of the cosets in $\Gamma/H$. 
We consider the set $W$ of pairs $(\tilde{\phi},\tilde{\psi})$  such that ${\rm lk}(\tilde{\phi}(\sigma), \tilde{\psi}(\{p,q\}))=1.$
Then $\Gamma$ acts on $W$ by setting $b\cdot (\tilde{\phi},\tilde{\psi})=(\tilde{\phi}'\in [b \tilde{\phi}], \tilde{\psi}'\in [b \tilde{\psi}]).$
Since $\Gamma$ acts freely on $X$, the same holds for the action of $\Gamma$ on $W$. Hence $[\delta_\Pi \times \delta_{[\gamma]}](X/\Gamma)$ is  the number of elements of $W\setminus \Gamma$.

We define a map $z:W\setminus \Gamma \rightarrow Z$. For each $r\in W\setminus \Gamma$, we choose an element  $(\tilde{\phi}(r),\tilde{\psi}(r))\in r$.
Since ${\rm lk}(\tilde{\phi}(r)(\sigma),\tilde{\psi}(r)(\{p,q\}))=1$, we can choose a point $\tilde{z}(r)\in \tilde{\phi}(r)(D) \cap \tilde{\psi}(r)(\tilde{\alpha})$.
Let $\tilde{x}(r)\in D$ and $\tilde{y}(r)\in \tilde{\alpha}$ such that $\tilde{\phi}(r)(\tilde{x}(r))=\tilde{\psi}(r)(\tilde{y}(r))=\tilde{z}(r)$.  
If $\pi_\Sigma:D\rightarrow \Sigma$ and $\pi_\alpha:\tilde{\alpha}\rightarrow \alpha$ denote the projections, define
$$
z(r)=(x(r),y(r))=(\pi_\Sigma(\tilde{x}(r)),\pi_\alpha(\tilde{y}(r))) \in Z.
$$

We claim that the map $z$ is injective. 
Suppose $z(r_1)=z(r_2)$, for $r_1,r_2\in W\setminus \Gamma$.  Then $\pi_\Sigma(\tilde{x}(r_1))=\pi_\Sigma(\tilde{x}(r_2))$ and 
$\pi_\alpha(\tilde{y}(r_1))=\pi_\alpha(\tilde{y}(r_2))$. Hence there exist $g\in G$ and $h\in H$ such that $\tilde{x}(r_2)=g \tilde{x}(r_1)$ and 
$\tilde{y}(r_2)=h \tilde{y}(r_1)$. Denote by $p=\tilde{x}(r_1)$ and $q=\tilde{y}(r_1)$ as points in $\mathbb{H}^3$, and let
$\tilde{\phi}_i=\tilde{\phi}(r_i)$, $\tilde{\psi}_i=\tilde{\psi}(r_i)$ for $i=1,2$. Then $\tilde{\phi}_1(p)=\tilde{\psi}_1(q)$ and 
$\tilde{\phi}_2g(p)=\tilde{\psi}_2h(q).$ Hence there is $b \in \Gamma$ such that
$\tilde{\psi}_1^{-1}\tilde{\phi}_1=h^{-1}\tilde{\psi}_2^{-1}\tilde{\phi}_2g=b$. This implies  
$$
(\tilde{\phi}_2g\tilde{\phi}_1^{-1})\tilde{\psi}_1=\tilde{\phi}_2 g b^{-1}=\tilde{\psi}_2h.
$$
Thus  $(\tilde{\phi}_2g\tilde{\phi}_1^{-1})\cdot (\tilde{\phi}_1,\tilde{\psi}_1)=(\tilde{\phi}_2,\tilde{\psi}_2)$, which proves $r_1=r_2$. Since $z$ is an injective map, we have that
$$
I(\delta_\Pi,\delta_{[\gamma]})=[\delta_\Pi \times \delta_{[\gamma]}](X/\Gamma)=\# (W\setminus \Gamma)\leq \#Z=\# (\Sigma \cap \alpha).
$$
This proves (i).

Suppose that $\Sigma$ is totally geodesic and $\alpha$ is a closed geodesic for the hyperbolic metric. We claim that the map $z$ is surjective.
Let $z=(x,y)\in W$. Let $p\in \mathbb{H}^3$ such that $\pi_M(p)=i_\Sigma(x)=i_\alpha(y)$. Consider $\tilde{x}\in D$ and $\tilde{y}\in \tilde{\alpha}$ such that
$\pi_\Sigma(\tilde{x})=x$ and $\pi_\alpha(\tilde{y})=y$. Hence there exist $\phi,\psi\in \Gamma$ such that $\phi(\tilde{x})=p$ and $\psi(\tilde{y})=p$.
We can choose $\tilde{\phi},\tilde{\psi}$ such that $\phi=\tilde{\phi}g$ and $\psi=\tilde{\psi}h$ for $g\in G$, $h\in H$. Then $\tilde{\phi} g(\tilde{x})=\tilde{\psi}h(\tilde{y})=p$.
Hence $\tilde{\phi}(D)$ intersects $\tilde{\psi}(\tilde{\alpha})$. But a totally geodesic disk and a transversal geodesic line in the hyperbolic space intersect if and only if the linking
number of their boundaries is  equal to one. Therefore $(\tilde{\phi},\tilde{\psi})\in W$. Let $r=[(\tilde{\phi},\tilde{\psi})]\in W\setminus \Gamma$. Since $[(\tilde{\phi}(r),\tilde{\psi}(r))]=r$, it follows that  there are $b\in \Gamma$, $g'\in G$, and $h'\in H$, such that
$$
b \tilde{\phi}= \tilde{\phi}(r)g', \, \, \, b \tilde{\psi}=\tilde{\psi}(r)h'.
$$
Recall that $\tilde{z}(r)$ is an element of the intersection $\tilde{\phi}(r)(D)\cap \tilde{\psi}(r)(\tilde{\alpha})$. Necessarily $b(p)\in \tilde{\phi}(r)(D)\cap \tilde{\psi}(r)(\tilde{\alpha})$\mbox{ and so,} because the intersection is a point, it follows that $\tilde{z}(r)=b(p)$. Notice that $\tilde{\phi}(r)g'g(\tilde{x})=\tilde{\psi}(r)h'h(\tilde{y})=b(p)=\tilde{z}(r)$. Since $D$ and $\tilde{\alpha}$ are embedded, $\tilde{x}(r)=g'g(\tilde{x})$ and $\tilde{y}(r)=h'h(\tilde{y}).$ Hence
$$
z(r)=(\pi_\Sigma(g'g(\tilde{x})),\pi_\alpha(h'h(\tilde{y})))=(x,y).
$$
This proves $z$ is surjective, and hence by the previous argument $z$ is a bijection. 
Thus
$$
I(\delta_\Pi,\delta_{[\gamma]})=[\delta_\Pi \times \delta_{[\gamma]}](X/\Gamma)=\# (W\setminus \Gamma)= \# (\Sigma \cap \alpha),
$$
which proves (ii).
\end{proof}

\section{Proof of Theorem \ref{intersection.number.thm}}\label{proof.intersection}

 Let $g$ be a metric of negative sectional curvature on $M$. Let $\Sigma \in \Pi$ be a surface which induces the group $G < \Gamma$. Set $\sigma={\Lambda}(G)\in QC_\eta$ 
and let $D\subset \H^3$ be a  $G$-invariant disk with $\partial_{\infty}D=\sigma$ and $\Sigma=D/G$.  Let $\Omega\subset D$ be a fundamental domain for the projection $D\rightarrow \Sigma$. Consider $\mathcal{G}'$ the set of $u\in \mathcal{G}$ such that ${\rm lk}(\sigma,u)=1$ and the geodesic $\gamma_g(u)$ intersects $\Omega$. Then
the set of elements $u\in \mathcal{G}$ such that ${\rm lk}(\sigma,u)=1$ is contained in the union $\bigcup_{g\in G} g\mathcal{G}'$. We can choose a partition
of unity $\{f_i\}_i$ as in the proof of Proposition \ref{length.form}.

Then, using Fubini's theorem as before  and the $\Gamma$-invariance of $\lambda_g$,
\begin{eqnarray*}
[\delta_\Pi \times \lambda_g](X/\Gamma) &=& \sum_i [\delta_\Pi \times \lambda_g](f_i)= \sum_i (\delta_\Pi \times \lambda_g)(\tilde{f}_i)\\
&=&\sum_i (\sum_{\tilde{\phi}} \delta_{\tilde{\phi}(\sigma)} \times \lambda_g)(\tilde{f}_i)
=\sum_i \int_{\mathcal{G}} \sum_{\tilde{\phi}} \tilde{f}_i(\tilde{\phi}(\sigma),u)\, d\lambda_g(u)\\
&=&\sum_i \int_{\mathcal{G}} \sum_{\tilde{\phi}} \tilde{f}_i(\tilde{\phi}(\sigma),\tilde{\phi}(u))\, d\lambda_g(u).\\
\end{eqnarray*}
Hence
\begin{eqnarray*}
[\delta_\Pi \times \lambda_g](X/\Gamma)&\leq& \sum_i \int_{\bigcup_{g\in G} g\mathcal{G}'} \sum_{\tilde{\phi}} \tilde{f}_i(\tilde{\phi}(\sigma),\tilde{\phi}(u))\, d\lambda_g(u)\\
&\leq& \sum_i \sum_{g\in G} \int_{g\mathcal{G}'} \sum_{\tilde{\phi}} \tilde{f}_i(\tilde{\phi}(\sigma),\tilde{\phi}(u))\, d\lambda_g(u)\\
&=&\sum_i \sum_{g\in G} \int_{\mathcal{G}'} \sum_{\tilde{\phi}} \tilde{f}_i(\tilde{\phi}(g(\sigma)),\tilde{\phi}(g(u)))\, d\lambda_g(u)\\
&=&\sum_i \sum_{\lambda\in \Gamma} \int_{\mathcal{G}'}  \tilde{f}_i(\lambda(\sigma),\lambda(u))\, d\lambda_g(u)\\
&=&\sum_i  \int_{\mathcal{G}'}  f_i(P(\sigma,u))\, d\lambda_g(u).\\
\end{eqnarray*}
Therefore, using \eqref{santalo.formula}, we have
\begin{multline*}
[\delta_\Pi \times \lambda_g](X/\Gamma)=\lambda_g(\mathcal{G}')
\leq \int_{\mathcal{G}} \# \, (\gamma_g(u)\cap \Omega) \, d\lambda_g(u)\\
=\pi {\rm area}_g(\Omega)=\pi {\rm area}_g(\Sigma).
\end{multline*}
Since this is for any  $\Sigma \in \Pi$, it follows that $I(\delta_\Pi,\lambda_g)\leq \pi {\rm area}_g(\Pi).$

Suppose $I(\delta_\Pi,\lambda_g)= \pi {\rm area}_g(\Pi),$ and choose $\Sigma\in \Pi$ such that ${\rm area}_g(\Sigma)={\rm area}_g(\Pi)$. Then  $$\lambda_g(\mathcal{G}')= \int_{\mathcal{G}} \# \, (\gamma_g(u)\cap \Omega) \, d\lambda_g(u),$$ which implies $\# \, (\gamma_g(u)\cap \Omega)\leq 1$ for $\lambda_g$-almost every $u\in \mathcal{G}$.
 If $\Sigma$ is not totally geodesic for the  metric $g$, we can find a point $p$ in the interior of $\Omega$  with a nonvanishing principal curvature. A local construction gives a geodesic $\gamma$ in $(\H^3,g)$ so that, near $p$,  $\gamma$ intersects $\Omega$ transversely in two  points.  Thus we can find an open neighborhood $V$ of  $\partial_{\infty}\gamma$ in $\mathcal G$ so that for any $v\in V$  we have $\#(\gamma_g(v)\cap\Omega)\geq 2.$
Since open sets of $\mathcal{G}$ have positive $\lambda_g$ measure, this is a contradiction. Hence $\Sigma$ is totally geodesic as we wanted to show. This proves (i).

From \cite[Theorem 1]{sigmund}, there is  a sequence $\{\gamma_i\}_{i\in\N}$ of closed geodesics of $(M,g)$ which becomes equidistributed in the metric $g$ (see \eqref{equidis}). Thus by  Proposition \ref{convergence.geodesic}, Proposition \ref{continuity.form} and Proposition \ref{length.form}, 
\begin{align*}
(2\pi \text{\rm vol}_g(M))^{-1}I(\nu,\lambda_g)&=\lim_{i\to\infty}\frac{I(\nu,\delta_{[\gamma_i]})}{l_g(\gamma_i)}=\pi\lim_{i\to\infty}\frac{l([\gamma_i])}{l_g(\gamma_i)}.
\end{align*}
By \cite[Lemma 2.19]{schapira}, $i(g,\overline{g})=\lim_{i\to\infty}\frac{l([\gamma_i])}{l_g(\gamma_i)}$. This proves Theorem \ref{intersection.number.thm} (ii).

\section{Proof of Theorem \ref{product.entropy.thm}, Corollary \ref{totally.geodesic.entropy}, and Corollary \ref{area.distortion}}\label{proof.second.thm}

\noindent{\bf Proof of Theorem \ref{product.entropy.thm}:} From  \eqref{alternative.entropy.mu} with $\mu=\overline{\nu}$, there is a sequence $\Pi_i\in\mathcal S_{1/i}(M)$  such that $\nu_{\Pi_i}\rightarrow \overline{\nu}$ and
$$
E_{\overline{\nu}}(g)=2\lim_i\frac{\text{area}(\Pi_i)}{\text{area}_g(\Pi_i)}.
$$
It follows by Proposition \ref{convergence} that
$\text{area}(\Pi_i)^{-1}{\delta_{\Pi_i}}\rightarrow {(2\pi\text{\rm vol}(M))}^{-1}{\nu}$.
Thus using Proposition \ref{continuity.form} and Theorem \ref{intersection.number.thm}, we have
\begin{align*}
\liminf_i\frac{\text{area}_g(\Pi_i)}{\text{area}(\Pi_i)} &\geq \frac 1 \pi \liminf_i I\left(\frac{\delta_{\Pi_i}}{\text{area}(\Pi_i)},\lambda_g\right)\\
& =\frac{1}{2\pi^2\text{\rm vol}(M)} I(\nu,\lambda_g)
 \geq \frac{\text{\rm vol}_g(M)}{\text{\rm vol}(M)}i(g,\overline{g}).
\end{align*}
We have $h(\bar g)=2$ and so $i(g,\overline{g})\geq \frac{h_L(g)}{2}$ by Theorem 1.1 of \cite{knieper-stretch}. Hence $$E_{\overline{\nu}}(g)h_L(g)\text{\rm vol}_g(M)\leq 4\text{\rm vol}_{\overline{g}}(M).$$
Suppose $E_{\overline{\nu}}(g)h_L(g)\text{\rm vol}_g(M)=4\text{\rm vol}_{\overline{g}}(M).$
It suffices to show that   $Gr_2(M)$ is foliated by totally geodesic disks for the metric $g$ because if that is the case then it follows from the  Codazzi equations that the metric $g$ has constant sectional curvature.

Let $\Sigma_i \in \Pi$ be a surface which induces the group $G_i < \Gamma$ and such that ${\rm area}_g(\Sigma_i)={\rm area}_g(\Pi_i)$. 
Let $D_i\subset \H^3$ be a  $G_i$-invariant disk with  $\Sigma_i=D_i/G_i$, and $\Omega_i\subset D_i$ be a fundamental domain for the
projection $D_i \rightarrow \Sigma_i$.

We will use the notation introduced in Section \ref{integral.geometry}. Consider the set
 $T_i=\{u\in \mathcal G:\, \gamma_g(u)\cap \Omega_i\neq \emptyset\},$
and the  function
 $$F_i:\mathcal G\rightarrow \Z, \quad F_i(u)=\min\{\#(\gamma_g(u)\cap \Omega_i)-1,1\}.$$
Consider the  $G_i$-invariant function $f_i$ on $D_i$ such that 
 $$ f_i(x)=\frac{1}{2}\int_{p_1^{-1}(x)}F_i\circ p_2(x,v)\, |\langle v, N_{\Sigma_i}(x)\rangle |d\omega_x(v)$$
for every $x\in \Omega_i$.
 Then $f_i\geq 0$. Using \eqref{santalo.formula0} we have that
\begin{equation}\label{crofton}
\int_{\Omega_i} f_i\, dA_g=\int_{T_i} F_i(u)\#(\gamma_g(u)\cap \Omega_i)d\lambda_g(u).
 \end{equation}

Since
 $$
\lim_i\frac{\text{area}_g(\Pi_i)}{\text{area}(\Pi_i)}= \frac{1}{\pi} \lim_i I\left(\frac{\delta_{\Pi_i}}{\text{area}(\Pi_i)},\lambda_g\right),
$$
the proof of Theorem \ref{intersection.number.thm} (i) implies that
$$
\lim_{i\to\infty}\frac{\lambda_g(T_i)}{\text{area}_g(\Pi_i)}=\lim_{i\to\infty}\frac{1}{\text{area}_g(\Pi_i)}\int_{T_i}\#(\gamma_g(u)\cap \Omega_{i})d\lambda_g(u).
$$
Hence
\begin{equation}\label{asym.equality}
\lim_{i\to\infty}\frac{1}{\text{area}_g(\Pi_i)}\int_{T_i}\{\#(\gamma_g(u)\cap \Omega_{i})-1\}d\lambda_g(u)=0.
\end{equation}
For $u\in T_i$,   $F_i(u)\#(\gamma_g(u)\cap \Omega_i)\leq 2\big(\#(\gamma_g(u)\cap \Omega_i)-1\big)$. Thus we deduce from  \eqref{crofton} and \eqref{asym.equality} that
\begin{equation*}\label{average.crofton2}
\lim_{i\to\infty}\frac{1}{\text{area}_g(\Pi_i)}\int_{\Omega_i} f_i\,dA_g=0.
 \end{equation*}
Since the sequence of homotopy groups $\{\Pi_i\}_{i\in\N}$ becomes equidistributed, we proceed  as in the proof of  Corollary 6.2  
 of \cite{calegari-marques-neves} to obtain, after passing to a subsequence,  $\phi_i\in \Gamma$  such that 
\begin{itemize}
\item[(a)] $\Lambda(\phi_iG_i\phi_i^{-1})$ converges  in Hausdorff distance, as $i\to\infty$, to $\sigma\in QC_0$;
\item[(b)]  $\{\phi(\sigma):\phi\in\Gamma\}$ is dense in $QC_0$;
\item[(c)]  for any $R>0$,
\begin{equation}\label{contradictions}\lim_{i\to\infty}\int_{\phi_i(D_i)\cap B_R(0)}f_i\circ \phi_i^{-1}dA_g=0.\end{equation}
\end{itemize}
 There is  $R>0$ so that $\text{conv}(\Lambda(\phi_iG_i\phi_i^{-1}))$ intersects $B_R(0)$ for every $i\in\N$. The surface $\Sigma_i$  is a stable minimal surface  for the metric $g$, hence the sequence $\phi_i(D_i)$  satisfies  uniform curvature estimates \cite{schoen}. Thus  the continuity property of convex hulls discussed in Section \ref{convex.hull} and Theorem 3.1 of \cite{calegari-marques-neves} gives the existence of a minimal disk $D$ in $(\H^3,g)$ so that $\partial_{\infty}D=\sigma$ and $\phi_i(D_i)$ converges strongly to $D$, after passing to a subsequence, on compact subsets of $\H^3$. 
 
 We argue that $D$ is totally geodesic for the metric $g$. Suppose there exists $y\in D$ such that the principal curvatures of $D$ at $y$ are nontrivial.  Then as before  there is $u\in \mathcal{G}$ such that $\gamma_g(u)$ intersects $D$ transversally at both $y$ and at $y'$, $y'\neq y$, nearby. By the smooth convergence there is an open set   $V\subset \mathcal G$ with $u\in V$  such that, for sufficiently large $i$, $\gamma_g(v)$ intersects $\phi_i(D_i)$ transversally  at points close to $y,y'$ for every $v\in V$. The injectivity radius of the surfaces $\Sigma_i$ is uniformly bounded from below.  By choosing $y$ and $y'$ sufficiently close to each other, one can choose $s>0$ and the fundamental domains $\Omega_i$ such that 
 $\# (\gamma_g(v) \cap \phi_i(\Omega_i) \cap B_s(y))\geq 2$ for every sufficiently large $i$ and every $v\in V$.  By replacing $D_i$ for $\phi_i(D_i)$ 
 in the previous equalities we can suppose that $\phi_i$ is the identity map. Hence, by \eqref{santalo.formula0},
  \begin{eqnarray*}
 \int_{\Omega_i\cap B_s(y)}f_i\, dA_g
 &=&  \int_{\mathcal{G}} \# (\gamma_g(u) \cap  \Omega_i\cap B_s(y))F_i(u) \, d\lambda_g(u)\\
 &\geq &\int_{V} \# (\gamma_g(u) \cap  \Omega_i\cap B_s(y))F_i(u) \, d\lambda_g(u)
 \geq 2 \lambda_g(V).
 \end{eqnarray*}
 This contradicts (c), since $\lambda_g(V)>0$. Therefore $D$ is totally geodesic.

From (b), the orbit of  $\sigma$ under $\Gamma$ is dense in $QC_0$. Thus by taking limits we obtain that for any $\sigma\in QC_0$ there is  a totally geodesic disk $D_{\sigma}$ in $(\H^3,g)$ with $\partial_{\infty}D_\sigma=\sigma$. A totally geodesic disk is  uniquely determined by its boundary  at infinity. Hence $D_\sigma$ depends continuously on $\sigma$ and  $\phi(D_\sigma)=D_{\phi(\sigma)}$ for every $\phi\in\Gamma$.  

The  set $\{(\sigma,x)\in QC_0\times \H^3:x\in D_\sigma\}$ is a five-dimensional manifold and the map $$\Phi:\{(\sigma,x)\in QC_0\times \H^3:x\in D_\sigma\}\rightarrow Gr_2(\H^3), \quad\Phi(\sigma,x)=(x,T_xD_\sigma),$$ is continuous and $\Gamma$-equivariant. Hence it induces a continuous map
$$\Phi':\mathcal{R}(\Gamma)\rightarrow Gr_2(M),$$
where $\mathcal{R}(\Gamma)=\{(\sigma,x)\in QC_0\times \H^3:x\in D_\sigma\}/\Gamma.$
The totally geodesic condition  implies that $\Phi'$ is injective. Since both  $\mathcal{R}(\Gamma)$ and $Gr_2(M)$ are compact  five-dimensional  manifolds,  invariance of domain  implies that $\Phi$ is a homeomorphism. Hence there is a totally geodesic surface tangent to any  element in $Gr_2(M)$. This finishes the proof
of Theorem \ref{product.entropy.thm}.
\medskip

\noindent{\bf Proof of Corollary \ref{totally.geodesic.entropy}:} The next proposition combined with Theorem \ref{product.entropy.thm} implies a direction of Corollary \ref{totally.geodesic.entropy}.
\begin{prop}\label{no.totally.geodesic}
If $(M,\overline{g})$ does not have compact totally geodesic surfaces, then $E_{\overline{\nu}}(g)=E(g)$ for any metric $g$ on $M$.
\end{prop}

\begin{proof}
Let  $\Pi_i\in\mathcal S_{1/i}(M)$ be a sequence such that  $$A(g)=\lim_i\frac{\text{area}_g(\Pi_i)}{\text{area}(\Pi_i)}.$$
 Let $\Sigma_i\in \Pi_i$ be such that  ${\rm area}(\Sigma_i)={\rm area}(\Pi_i)$. We claim that the sequence $\{\Pi_i\}_{i\in\N}$ becomes equidistributed.

Let $F(\Sigma)$ denote the frame bundle of an orientable immersed surface $\Sigma$. The metric $\overline{g}$ induces a natural measure $d\theta$ on  $F(\Sigma)$ of total measure $2\pi\text{area}(\Sigma).$  Define the following measure on  $F(M)$, the frame bundle   of $M$:
$$
\nu_i(\psi)=\frac{1}{2\pi\text{area}(\Pi_i)}\int_{F(\Sigma_i)}\psi(x,\{e_1,e_2,n(x)\})\, d\theta(x,e_1,e_2),
$$
where $\psi:F(M)\rightarrow\mathbb{R}$ is a continuous function and  $n(x)$ is the normal vector to $\Sigma_i$ at $x$. After passing to a subsequence, we obtain that the sequence $\nu_i$ converges to a unit measure $\tilde{\nu}$ on $F(M)$ that is invariant under the homogeneous  $\rm{PSL}(2,\R)$-action (see \cite[Lemma 3.2]{lowe-neves}).  Using  that $M$ has no closed totally geodesic surfaces, we obtain from Ratner's classification theorem \cite{ratner}  that $\tilde{\nu}$ is the homogeneous unit measure on $F(M)$. Hence $\nu_{\Pi_i}\rightarrow \overline{\nu}$, which implies that  the sequence $\{\Pi_i\}_{i\in\N}$ equidistributes. By (\ref{alternative.entropy.mu}) and Proposition \ref{alternative.definition},  $E_{\overline{\nu}}(g)\geq E(g)$ and hence $E_{\overline{\nu}}(g)=E(g)$.
 
\end{proof}

The next proposition then finishes the proof of Corollary \ref{totally.geodesic.entropy}.
\begin{prop}\label{not.true}
Suppose $(M,\overline{g})$ has compact totally geodesic surfaces. Then  there is a sequence of metrics $g_i$ on $M$ which converges to  $\overline{g}$ smoothly and such that
$$E(g_i)h_L(g_i)\text{\rm vol}_{g_i}(M)> 4\text{\rm vol}_{\overline{g}}(M)$$
for any $i$.
\end{prop}

\begin{proof}
Let $i:\Sigma \rightarrow (M,\overline{g})$ be a compact totally geodesic surface inducing the surface group $G< \Gamma$, and let $\Pi$ be the conjugacy class of $G$.  Then $\Pi\in S_0(M)$, since the surface $\Sigma$ lifts to a totally geodesic disk in $\H^3$. Therefore 
$A(g)\leq \frac{\text{\rm area}_g(\Pi)}{\text{\rm area}(\Pi)}$ for any metric $g$ on $M$ and so by  Proposition \ref{alternative.definition}, $E(g)\geq 2\frac{\text{\rm area}(\Pi)}{\text{\rm area}_g(\Pi)}$ for any $g$.

Suppose by contradiction that the proposition is false. Then there is an open set $\mathcal{V}$ of Riemannian metrics in the smooth topology with $\overline{g}\in \mathcal{V}$ such that  $E(g)h_L(g)\text{\rm vol}_{g}(M)\leq 4\text{\rm vol}_{\overline{g}}(M)$
for any $g\in \mathcal{V}$. Combining with $E(g)\geq 2\frac{\text{\rm area}(\Pi)}{\text{\rm area}_g(\Pi)}$ we have that
$$
\frac{\text{\rm area}(\Pi)}{\text{\rm area}_g(\Pi)}h_L(g)\text{\rm vol}_{g}(M)\leq 2\text{\rm vol}_{\overline{g}}(M)
$$
and so
$$
\frac{\text{\rm area}(\Pi)}{2\text{\rm vol}_{\overline{g}}(M)}h_L(g)\text{\rm vol}_{g}(M)\leq \text{\rm area}_g(\Pi) 
$$
for any $g\in \mathcal{V}$, with equality at $g=\overline{g}$. Since $h_L(c^2g)=c^{-1}h_L(g)$, 
\begin{equation}\label{differentiate}
\frac{\text{\rm area}(\Pi)}{2\text{\rm vol}_{\overline{g}}(M)^\frac23}
h_L\big(\frac{\text{\rm vol}_{\overline{g}}(M)^\frac23}{\text{\rm vol}_g(M)^\frac23}g\big)
\text{\rm vol}_{g}(M)^\frac23\leq \text{\rm area}_g(\Pi) 
\end{equation}
for  $g\in \mathcal{V}$, with equality at $g=\overline{g}$.  

Let $U$ be an open set of $M$ which is disjoint from $\Sigma$. Consider $f:M\rightarrow \mathbb{R}$ a smooth nonnegative function  with support contained in $U$ such that $f$ is positive somewhere. Define the variation of metrics $g(t)=(1+tf)^2\overline{g}$ such that $g(0)=\overline{g}$.  Since $\Sigma$ is the unique minimal surface in its homotopy class, it follows that $\text{area}(\Sigma)=\text{area}(\Pi)$.  For $t\geq 0$, $g(t)\geq \overline{g}$ on $M$. Since 
$\text{area}_{g(t)}(\Sigma)=\text{area}_{\overline{g}}(\Sigma)$, it follows that  ${\rm area}_{g(t)}(\Pi)=\text{area}(\Sigma)$ for $t\geq 0$.  Therefore, by (\ref{differentiate}) with $g=g(t)$,
$$
\frac{d}{dt}_{|t=0} 
\big(h_L\big(\tilde{g}(t)\big)
\text{\rm vol}_{g(t)}(M)^\frac23\big)\leq 0,
$$
where $\tilde{g}(t)=\frac{\text{\rm vol}_{\overline{g}}(M)^\frac23}{\text{\rm vol}_{g(t)}(M)^\frac23}g(t)$. Notice that $\text{\rm vol}\big(\tilde{g}(t)\big)=\text{\rm vol}(\overline{g})$ for any $t$. This is a contradiction, since the hyperbolic metric is a critical point of the Liouville entropy when restricted
to metrics of the same volume (\cite{katok-knieper-weiss}),  and $\frac{d}{dt}_{|t=0} \text{\rm vol}_{g(t)}(M)>0$. This proves  the  proposition.
\end{proof}

\noindent{\bf Proof of Corollary \ref{area.distortion}:}
As before, by  (\ref{alternative.entropy.mu}) with $\mu=\overline{\nu}$, there is a sequence $\Pi_i\in\mathcal S_{1/i}(M)$  such that $\nu_{\Pi_i}\rightarrow \overline{\nu}$ and
$$
E_{\overline{\nu}}(g)=2\lim_i\frac{\text{area}(\Pi_i)}{\text{area}_g(\Pi_i)}.
$$

We argue like in \cite[Proposition 6.3]{lowe-neves}. Let $\Sigma_i\in \Pi_i$ be such that ${\rm area}(\Sigma_i)={\rm area}(\Pi_i)$.   Since 
$$
{\rm area}_g(\Sigma_i)=\int_{\Sigma_i} \frac{{\rm area}_g}{{\rm area}_{\overline{g}}}(x,T_x\Sigma_i)(d\Sigma_i)_{\overline{g}}(x),
$$
and the sequence $\{\Pi\}_i$ becomes equidistributed, it follows that
$$
\lim_i\frac{{\rm area}_g(\Sigma_i)}{{\rm area}(\Sigma_i)}=\frac{1}{{\rm vol}(Gr_2(M))}\int_{Gr_2(M)} \frac{{\rm area}_g}{{\rm area}_{\overline{g}}}(x,P)dV_{\overline{g}}=A(g/\bar g).
$$
Since  ${\rm area}_g(\Pi_i)\leq {\rm area}_g(\Sigma_i)$ we deduce $A(g/\bar g)E_{\overline{\nu}}(g)\geq 2$ and so, by
 Theorem \ref{product.entropy.thm},
$$
 h_L(g) {\rm vol}_g(M)\leq 2 A(g/\bar g)\text{\rm vol}_{\overline{g}}(M).
$$
If equality holds then  $A(g/\bar g)E_{\overline{\nu}}(g)=2$ and $g$ has constant sectional curvature.  This implies    $E_{\overline{\nu}}(g)=E(g)$ and thus  $A(g/\bar g)E(g)=2$. Theorem 1.2 in \cite{lowe-neves} shows that $g$ must be a multiple of $\bar g$.

\section{Mostow Rigidity Theorem}\label{mostow.rigidity}
 In this section we prove:
 \begin{thm}[Mostow Rigidity]\label{mostow} Consider two closed hyperbolic manifolds   $M=\H^3/\Gamma_M$ and $N=\H^3/\Gamma_N$ for which there is   a homotopy equivalence $F:N\rightarrow M$.  Then $F$ is homotopic to an isometry $G:N\rightarrow M$.
 \end{thm}

Let $E:M\rightarrow N$ denote the homotopic inverse of $F$. The lifts $F:\H^3\rightarrow \H^3$, $E:\H^3\rightarrow \H^3$ can be assumed to be $C^1$ and both maps are quasi-isometries (see \cite[Chapter C]{benedetti} for instance). They induce  homeomorphisms on the boundary at infinity $\tilde{F}:S^2\rightarrow S^2$, $\tilde{E}:S^2\rightarrow S^2$, with  $\tilde{E}\circ \tilde{F}=\tilde{F}\circ \tilde{E}= id_{S^2}.$ The maps $\tilde{E},\tilde{F}$ are quasi-conformal  and thus map quasi-circles to quasi-circles. The maps  $F_{*}:\Gamma_N\rightarrow \Gamma_M$ and $E_{*}:\Gamma_M\rightarrow \Gamma_N$ are inverses of each other. For  $\phi\in\Gamma_N$, we have $F_*(\phi)\circ F=F\circ \phi$ and $F_*(\phi)\circ \tilde{F}=\tilde{F}\circ \phi$. For $\psi\in \Gamma_M$, $E_*(\psi)\circ E=E\circ \psi$ and $E_*(\psi)\circ \tilde{E}=\tilde{E}\circ \psi$.

If $\Pi<\Gamma_M$ is a homotopy class of essential surfaces and $G\in \Pi$, then $E_*(G)$ is  a  surface group of $\Gamma_N$ whose  conjugacy class coincides  with $E_*(\Pi)=\{E_*(G):G\in \Pi\}$. If $\Pi$ is quasi-Fuchsian, then $E_*(\Pi)$ is also quasi-Fuchsian since $\tilde{E}$ is quasi-conformal. The map $\tilde{F}$ induces a homeomorphism on the space of all geodesics of $\H^3$.  Hence  if $\lambda\in \mathcal C(\Gamma_N)$ is a  geodesic current, then  $(\tilde{F})_{*}\lambda\in \mathcal C(\Gamma_M)$. Notice that
 if $\Pi_N<\Gamma_N$ is a homotopy class of quasi-Fuchsian surfaces, then $(\tilde{F})_*\delta_{\Pi_N}=\delta_{F_*(\Pi_N)}.$

We claim that the proof of Theorem \ref{intersection.number.thm} (i) can be extended to show that for any negatively curved metric $g$ on $N$ and every quasi-Fuchsian homotopy class  $\Pi<\Gamma_M$, we have
\begin{equation}\label{new.4}
I_{\Gamma_M}(\delta_{\Pi},(\tilde{F})_*\lambda_g)\leq \pi\text{area}_g(E_*(\Pi)).
\end{equation}
Consider $\{f_i\}_i$ as the partition of unity  of the proof of Proposition \ref{length.form}.
Since $\delta_\Pi=(\tilde{F})_*\delta_{E_*(\Pi)}$, it follows as in the proof of Theorem \ref{intersection.number.thm} (i),
\begin{eqnarray*}
&&[\delta_\Pi \times (\tilde{F})_*\lambda_g](X/\Gamma_M) = \sum_i [\delta_\Pi \times (\tilde{F})_*\lambda_g](f_i)= \sum_i (\delta_\Pi \times (\tilde{F})_*\lambda_g)(\tilde{f}_i)\\
&&= \sum_i ((\tilde{F})_*\delta_{E_*(\Pi)} \times (\tilde{F})_*\lambda_g)(\tilde{f}_i)\\
&&= \sum_i (\delta_{E_*(\Pi)} \times \lambda_g)(\tilde{f}_i\circ \tilde{F})=[\delta_{E_*(\Pi)} \times \lambda_g](X/\Gamma_N). \\
\end{eqnarray*}
Hence $I_{\Gamma_M}(\delta_{\Pi},(\tilde{F})_*\lambda_g)=I_{\Gamma_N}(\delta_{E_*(\Pi)},\lambda_g)$. Therefore (\ref{new.4}) follows from  Theorem \ref{intersection.number.thm} (i).

The next proposition follows by adapting the construction made in  \cite[Theorem 2.12]{knieper} using Busemann functions instead of orthogonal projections. If  the map $F$ is the identity, a map $\chi$ can be constructed as in \cite[Section 2.4]{schapira}. Let us denote by $\{g^t\}_{t\in \mathbb{R}}$ the geodesic flow of a Riemannian metric $g$.

\begin{prop}\label{gromov-correspondence} Consider two negatively curved metrics $(N,g_1)$ and $(M,g_2)$. There is a homeomorphism  $\chi:T^1_{g_1}(\H^3)\rightarrow T^1_{g_2}(\H^3)$ so that
\begin{itemize}
\item[(i)]$F_{*}(\phi)(\chi(v))=\chi(\phi(v))$ for any $\phi\in \Gamma_N$ and hence $\chi$ induces  a homeomorphism $$\chi:T^1_{g_1}(N)\rightarrow T^1_{g_2}(M);$$
\item[(ii)] if $\gamma_v$ and $\sigma_{\chi(v)}$ are the unique unit speed geodesics  generated by $v$ and $\chi(v)$, then 
$\tilde{F}(\partial_{\infty}\gamma_v)=\partial_{\infty}\sigma_{\chi(v)}$;
\item[(iii)] there is a H\"older continuous function $T:\R\times T^1_{g_1}(\H^3)\rightarrow \R$ which is $C^1$ along the $t$-direction and is such that $g_{2}^{T(t,v)}(\chi(v))=\chi(g_1^t(v))$ for any $v\in T^1_{g_1}(\H^3)$, $t\in\R$;
\item[(iv)] the function $\Psi(v)=\frac{d}{dt}T(0,v)>0$ is  H\"older continuous;
\item[(v)] there is a constant $C>0$ such that $d_{g_2}(\pi(\chi(v)),F(\pi(v)))\leq C$ for any $v\in T^1_{g_1}(\H^3)$, where $\pi(v)$ denotes the basepoint of $v$.
\end{itemize}
\end{prop}

\begin{proof}
Let  $B_{\xi}(x,y)$ denote the  Busemann function for $(\H^3,g_2),$ where $\xi\in S^2$ and $x,y\in \H^3$. For  $v\in T^1_{g_1}(\H^3)$, let $\gamma_v$ be the unit speed geodesic for $(\H^3,g_1)$ with $(\gamma_v(0),\gamma'_v(0))=v$. Let $\sigma$ be the unit speed geodesic for $(\H^3,g_2)$ with $\partial_\infty\sigma=\tilde{F}(\partial_\infty\gamma_v)$.  There is a unique $\tilde t$ so that $B_{\tilde{F}(\gamma_v(+\infty))}(F(\pi(v)), \sigma(\tilde t))=0$. We define $\tilde{\chi}(v)=(\sigma(\tilde t),\sigma'(\tilde t))$. The map $\tilde{\chi}$ is continuous,  surjective, and satisfies 
$\tilde{\chi} \circ \phi=F_*(\phi)\circ \tilde{\chi}$ for any $\phi \in \Gamma_N$. Using  that the  Busemann function satisfies the co-cycle condition
$$B_{\xi}(x,y)=B_{\xi}(x,z)+B_{\xi}(z,y),$$
the same argument as in the proof of Lemma 2.12  \cite{schapira} shows that if we set $$s(t,v)=B_{\tilde{F}(\gamma_v(+\infty))}(F(\gamma_v(0)), F(\gamma_v(t))),$$ then 
$\tilde{\chi}(g_1^t(v))=g_{2}^{s(t,v)}(\tilde{\chi}(v))$ for every $t\in\R$ and $v\in T^1_{g_1}(\H^3)$. 
The function $s$ is H\"older continuous, differentiable along the $t$-direction, and $v\mapsto \partial_ts(t,v)$ is H\"older continuous.

The map $\tilde{\chi}$ may not be injective. Since $F$ is a quasi-isometry, there is $\tau>0$  such  that  $s(\tau,v)>0$ for any $v\in T^1_{g_1}(\H^3)$. Consider the function
$$r(v)=\frac{1}{\tau}\int_0^\tau s(t,v)dt,\quad v\in T^1_{g_1}(\H^3).$$
Following \cite[Theorem 2.12]{knieper}, we define 
$$\chi:T^1_{g_1}(\H^3)\rightarrow T^1_{g_2}(\H^3),\quad \chi(v)=g_2^{r(v)}(\tilde{\chi}(v)).$$
The fact that $\chi$ is a homeomorphism follows just as in  \cite{knieper}.
The function $T$ can be defined by $T(t,v)=r(g^t_1(v))+s(t,v)-r(v).$

The statements (i)-(iv) of the proposition follow from the construction. We will denote by $C$ a positive constant independent of $t,v$ that depends on the line. To check property $(v)$, notice that since $F$ is a quasi-isometry we have  $$|s(t,v)|\leq d_{g_2}(F(\gamma_v(0)),F(\gamma_v(t)))\leq C|t|+C.$$ Hence  $|r(v)|\leq C$ for every  $v\in T^1_{g_1}(\H^3)$ and it suffices to check  (v) for the map $\tilde{\chi}$. 

In the definition of  $\tilde{\chi}$,  assume that $\tilde t=0$ so $\tilde{\chi}(v)=(\sigma(0),\sigma'(0))$. From Morse's Lemma  it follows  that the Hausdorff distance between $F\circ\gamma_v$ and $\sigma$ is finite. Then $d_{g_2}(F(\gamma_v(0)),\sigma(\bar t))\leq C$ for some $\bar t\in\R$. If $\xi=\tilde{F}(\gamma_v(+\infty))$ and using that $B_{\xi}(\sigma(0),F(\gamma_v(0))=0$, 
\begin{align*}
d_{g_2}(\sigma(0),F(\gamma_v(0)))&\leq  d_{g_2}(\sigma(0),\sigma(\bar t))+d_{g_2}(\sigma(\bar t),F(\gamma_v(0)))\\
&\leq d_{g_2}(\sigma(0),\sigma(\bar t))+C=
|B_{\xi}(\sigma(0),\sigma(\bar t))|+C\\
&\leq |B_{\xi}(\sigma(0),F(\gamma_v(0)))|+|B_{\xi}(\sigma(\bar t),F(\gamma_v(0)))|+C\\
&=|B_{\xi}(\sigma(\bar t),F(\gamma_v(0)))|+C\\
&\leq d_{g_2}(F(\gamma_v(0)),\sigma(\bar t))+C\leq 2C.
\end{align*}
This finishes the proof  of the proposition.
\end{proof}

The geodesic stretch can be determined by the lengths of periodic orbits (Section 2 of \cite{schapira}). The proof of the next proposition is based on \cite[Theorem 1.4]{fathi}.  
\begin{prop}\label{quotient.geodesics} Let  $(N,g_1)$ and $(M,g_2)$ be negatively curved. Suppose $\{\gamma_i\}_{i\in\N}$ is a sequence of closed geodesics in $(N,g_1)$ equidistributed for the metric $g_1$. Then
$$i_F(g_1,g_2)=\lim_{i\to\infty}\frac{l_{g_2}(F_*[\gamma_i])}{l_{g_1}([\gamma_i])}.$$
\end{prop}
\begin{proof}
Consider  $\chi,T,\Psi$ as in Proposition \ref{gromov-correspondence}. The functions $T$ and $\Psi$ are equivariant for the isomorphism $F_*:\Gamma_N\rightarrow \Gamma_M$, and hence  induce functions  $T:\R\times T^1_g(N)\rightarrow \mathbb{R}$ and $\Psi:T^1_g(N)\rightarrow \mathbb{R}$.

Let  $\sigma_i$ be the closed geodesic of $(M,g_2)$ in the conjugacy class $F_{*}[\gamma_i]$. Let $v_i\in T^1_{g_1}(\H^3)$ so that  $\tilde\gamma_i(t)=\pi(g_1^t(v_i))$ is a unit speed geodesic in $(\H^3,g_1)$ that projects to $\gamma_i$ in  $N$ under the quotient map. Then $\tilde \sigma_i(s)=\pi(g_2^s(\chi(v_i)))$ is a unit speed geodesic that projects to $\sigma_i$ in $M$. Since $\chi(g_1^t(v))=g_2^{
T(t,v)}(\chi(v))$, it follows that
\begin{equation}\label{length}
l_{g_2}(\sigma_i)=\int_0^{l_{g_1}(\gamma_i)}|\frac{d}{dt}\tilde\sigma_i(T(t,v_i))|dt=\int_0^{l_{g_1}(\gamma_i)}\Psi(\gamma_i(t),\gamma_i'(t))dt.
\end{equation}
Thus we obtain from  \eqref{equidis} that
\begin{eqnarray}\label{equidis1}
\nonumber\lim_{i\to\infty}\frac{l_{g_2}(F_*[\gamma_i])}{l_{g_1}([\gamma_i])}&=&\lim_{i\to\infty}\frac{1}{l_{g_1}(\gamma_i)}\int_0^{l_{g_1}(\gamma_i)}\Psi(\gamma_i(t),\gamma_i'(t))dt\\
&=&\frac{1}{\text{\rm vol}_{g_1}(T^1_{g_1}(N)}\int_{T^1_{g_1}(N)}\Psi dV_{g_1}.
\end{eqnarray}

For $v\in T^1_{g_1}(N)$, consider $\tilde{v}\in T^1_{g_1}(\H^3)$ which projects to $v$. We have 
$a(t,v)=d_{g_2}(F(\pi(g_1^t(\tilde v))),F(\pi(\tilde{v})))$, and from Proposition \ref{gromov-correspondence} (iii) $$T(t,v)=T(t,\tilde{v})=d_{g_2}(\pi(\chi(g_1^t(\tilde v))),\pi(\chi(\tilde v))).$$ 
From Proposition \ref{gromov-correspondence} (v), there is a positive constant $C$ so that $|a(t,v)-T(t,v)|\leq C$ for any  $t\in \R$ and $v\in T^1_{g_1}(N)$. Hence
\begin{equation}\label{last.identity}
 i_F(g_1,g_2)=\frac{1}{\text{\rm vol}(T^1_{g_1}(N))}\lim_{t\to\infty}\int_{T^1_{g_1}(N)}\frac{T(t,v)}{t}dV_{g_1}(v).
 \end{equation}
From Proposition \ref{gromov-correspondence} (iii),  $$T(t+s,v)=T(t,g_1^s(v))+T(s,v)$$ and hence $\frac{d}{dt}T(t,v)=\Psi(g_1^t(v))$. Therefore
$$\int_{T^1_{g_1}(N)}\frac{T(t,v)}{t}dV_{g_1}(v)=\int_{T^1_{g_1}(N)}\frac{1}{t}\int_0^t\Psi(g_1^s(v))ds dV_{g_1}(v)=\int_{T^1_{g_1}(N)}\Psi dV_{g_1},$$
where we used that  $dV_{g_1}$ is invariant under the geodesic flow. The proposition follows by combining this  with \eqref{equidis1} and \eqref{last.identity}.
\end{proof}

The next proposition was proven in \cite{knieper-stretch} when $F$ is the identity map. The proof  presented here uses Abramov's formula. This generalizes Theorem \ref{intersection.number.thm} (ii).
\begin{prop}\label{stretch.entropy} Let $g$ be a metric of negative sectional curvature  on $N$. Then
$$I_{\Gamma_M}(\nu,(\tilde{F})_{*}\lambda_g)=2\pi^2\text{\rm vol}_g(N)i_F(g,\overline{g}_M)\geq\pi^2 \text{\rm vol}_g(N)h_L(g).$$
\end{prop}

\begin{proof}

Recall the  homeomorphism $\chi:T^1_g(\H^3)\rightarrow T^1(\H^3)$ given by Proposition \ref{gromov-correspondence} with $g_1=g$ and $g_2=\overline{g}_M$.  Consider the functions $T$ and $\Psi$ provided by that proposition. We  denote the induced homeomorphism on the quotients by  $\chi:T^1_g(N) \rightarrow T^1_{\overline{g}_M}(M)$. The functions $T$ and $\Psi$ are equivariant for the isomorphism $F_*:\Gamma_N\rightarrow \Gamma_M$, and hence  induce functions  $T:\R\times T^1_g(N)\rightarrow \mathbb{R}$ and $\Psi:T^1_g(N)\rightarrow \mathbb{R}$. 

From \cite{sigmund}, there is a sequence $\{\gamma_i\}_{i\in\N}$ of closed geodesics of $(N,g)$ which becomes equidistributed for the metric $g$ (see \eqref{equidis}). The Proposition \ref{convergence.geodesic} implies that $\frac{1}{l_g(\gamma_i)}\delta_{[\gamma_i]}\rightarrow \frac{1}{2\pi \text{\rm vol}_g(N)}\lambda_g$.  As in the case of surfaces, $\tilde{F}_*(\delta_{[\gamma_i]})=\delta_{F_*([\gamma_i])}$.  Therefore
$\frac{1}{l_g(\gamma_i)}\delta_{F_*([\gamma_i])}\rightarrow \frac{1}{2\pi \text{\rm vol}_g(N)}\tilde{F}_*(\lambda_g),$ and hence by Proposition \ref{continuity.form},  
$$
(2\pi \text{\rm vol}_g(N))^{-1}I_{\Gamma_M}(\nu,(\tilde{F})_*\lambda_g)=\lim_{i\to\infty}\frac{I_{\Gamma_M}(\nu,\delta_{F_{*}[\gamma_i]})}{l_g(\gamma_i)}.
$$
By Proposition \ref{length.form} and Proposition \ref{quotient.geodesics},
$$
(2\pi \text{\rm vol}_g(N))^{-1}I_{\Gamma_M}(\nu,(\tilde{F})_*\lambda_g)=\pi \lim_{i\to\infty}\frac{l_{\overline{g}_M}(F_*([\gamma_i]))}{l_g(\gamma_i)}=\pi i_F(g,\overline{g}_M).
$$
This proves $I_{\Gamma_M}(\nu,(\tilde{F})_{*}\lambda_g)=2\pi^2\text{\rm vol}_g(N)i_F(g,\overline{g}_M)$.

 Consider  the flow $\{U_s\}_{s\in\R}$ on $T^1_g(N)$ defined by $U_s(v)=\chi^{-1}(\overline{g}_M^s(\chi(v)))$. This corresponds to a reparametrization  of the geodesic flow on $T^1_g(N)$. For  a unit measure $\mu$ on $T^1_g(N)$  that is invariant under  the flow $U=\{U_s\}_{s\in\R}$, the entropy of $U$ with respect to $\mu$ is denoted by $h_\mu(U)$. Let $\tau:\R\times T^1_g(N)\rightarrow \R$ be such that $U_s(v)=g^{\tau(s,v)}(v)$.  Then
 $$T(\tau(s,v),v)=\tau(T(s,v),v)$$
for any  $s\in\R, v\in  T^1_g(N$).
 The function $\tau$ is continuous and  differentiable along the $s$-direction. Define  a unit measure $\mu_\Psi$ on $T^1_g(N)$   by
 $$\mu_\Psi(\psi)=\frac{\int_{T^1_g(N)}\psi\Psi dV_g}{\int_{T^1_g(N)}\Psi dV_g}, \quad \psi\in C(T^1_g(N)).$$
Since  $\Psi(v)\frac{d}{ds}\tau(0,v)=1$ for $v\in T^1_g(N)$, the measure $\mu_{\Psi}$ is preserved  by $\{U_s\}_{s\in\R}$. From Abramov's formula \cite{abramov}, 
\begin{equation}\label{abramov}
\frac{h_L(g)}{h_{\mu_\Psi}(U)}=\frac{1}{\text{\rm vol}(T^1_g(N))}\int_{T^1_g(N)}\Psi dV_g.
\end{equation}

Notice that it follows from the  proof of Proposition \ref{quotient.geodesics}  that
\begin{equation}\label{intersection.identity}
i_F(g,\overline{g}_M)=\frac{1}{\text{\rm vol}(T^1_g(N))}\int_{T^1_g(N)}\Psi dV_g.
\end{equation}
Hence $h_L(g)=i_F(g,\overline{g}_M) h_{\mu_\Psi}(U).$
The flow $U=\{U_s\}$ is topologically conjugate to the geodesic flow of $\overline{g}_M$ on $M$ by the map $\chi$. Hence the topological entropy $h_{top}(U)$ of  $\{U_s\}_{s\in\R}$ is equal to $2$. The variational principle implies  that $h_{\mu_\Psi}(U)\leq h_{top}(U)=2$. Therefore $h_L(g)\leq 2i_F(g,\overline{g}_M)$, which finishes the proof of the proposition.

\end{proof}

\begin{proof}[Proof of Theorem \ref{mostow}]
 Let us  suppose that $\text{\rm vol}_{\overline{g}_M}(M)\leq \text{\rm vol}_{\overline{g}_N}(N)$.
Consider a sequence  $\{\Pi_i\in S_{\eta_i}(M)\}_{i\in\N}$ of homotopy classes of quasi-Fuchsian surfaces in $M$, $\eta_i\rightarrow 0$,  which becomes equidistributed in the sense described in \eqref{equidis.defi}. Such sequences can  be constructed as in Proposition 6.1 of \cite{lowe-neves}. Since $\text{area}_{\overline{g}_M}(\Pi_i)^{-1}{\delta_{\Pi_i}}\rightarrow {(2\pi\text{\rm vol}_{\overline{g}_M}(M))}^{-1}{\nu}$ by Proposition \ref{convergence}, it follows by \eqref{new.4} and  Proposition \ref{continuity.form} that  
\begin{eqnarray*}
\frac{I_{\Gamma_M}(\nu,(\tilde{F})_*\lambda_{\overline{g}_N})}{2\pi^2\text{\rm vol}_{\overline{g}_M}(M)}&=&\frac 1 \pi \lim_i I_{\Gamma_M}\left(\frac{\delta_{\Pi_i}}{\text{area}_{\overline{g}_M}(\Pi_i)},(\tilde{F})_*\lambda_{\overline{g}_N}\right)\\
&\leq&\liminf_i\frac{\text{area}_{\overline{g}_N}(E_*(\Pi_i))}{\text{area}_{\overline{g}_M}(\Pi_i)}.
\end{eqnarray*}
Since $h_L(\overline{g}_N)=2,$  by Proposition \ref{stretch.entropy} it follows that 
$I_{\Gamma_M}(\nu,(\tilde{F})_{*}\lambda_{\overline{g}_N})\geq 2\pi^2 \text{\rm vol}_{\overline{g}_N}(N)$ and hence
\begin{equation}\label{lower.bound.mostow}
1\leq  \frac{\text{\rm vol}_{\overline{g}_N}(N)}{\text{\rm vol}_{\overline{g}_M}(M)}\leq \liminf_i\frac{\text{area}_{\overline{g}_N}(E_*(\Pi_i))}{\text{area}_{\overline{g}_M}(\Pi_i)}.
\end{equation} 

Let $\tilde{\Sigma}_i\in E_*(\Pi_i)$  be such that  ${\rm area}_{\overline{g}_N}(\tilde{\Sigma}_i)={\rm area}_{\overline{g}_N}(E_*(\Pi_i))$, and let $g_i$ be the genus of $\tilde{\Sigma}_i$.
Then as in the proof of Proposition \ref{alternative.definition},
\begin{equation}\label{gauss.mostow}
\frac{1}{\text{area}_{\overline{g}_N}(\tilde{\Sigma}_i)}\int_{\tilde{\Sigma}_i}\frac{|A_{\tilde{\Sigma}_i}|^2}{2}dA_{\overline{g}_N}=\frac{4\pi({g_i}-1)}{\text{area}_{\overline{g}_N}(\tilde{\Sigma}_i)}-1,
 \end{equation}
where $A_{\tilde{\Sigma}_i}$ is the second fundamental form of $\tilde{\Sigma}_i$ in the metric $\overline{g}_N$. Let $\Sigma_i\in \Pi_i$ be such that ${\rm area}_{\overline{g}_M}(\Sigma_i)={\rm area}_{\overline{g}_M}(\Pi_i)$. Since the genus of $\Sigma_i$ is $g_i$,  
\begin{equation}\label{seppi.genus}
\frac{1}{\text{area}_{\overline{g}_M}(\Sigma_i)}\int_{\Sigma_i}\frac{|A_{\Sigma_i}|^2}{2}dA_{\overline{g}_M}=\frac{4\pi({g_i}-1)}{\text{area}_{\overline{g}_M}(\Sigma_i)}-1,
 \end{equation}
 where $A_{\Sigma_i}$ is the second fundamental form of $\Sigma_i$ in the metric $\overline{g}_M$.
By \cite{seppi}, $|A_{\Sigma_i}|\leq C \ln (1+\eta_i)$.
  Combining  \eqref{lower.bound.mostow}, \eqref{gauss.mostow} and (\ref{seppi.genus}),
 \begin{equation}\label{thm6.1.condition}
\lim_{i\to\infty} \frac{1}{\text{area}_{\overline{g}_N}(\tilde{\Sigma}_i)}\int_{\tilde{\Sigma}_i}|A_{\tilde{\Sigma}_i}|^2dA_{\overline{g}_N}=0.
\end{equation}

We will use  Theorem 6.1 of \cite{calegari-marques-neves},  which relies on the Ratner-Shah \cite{ratner,shah} classification theorem in homogeneous dynamics, to prove  the map $\tilde{E}$ sends round circles to round circles.
Consider   an $E_*(G_i)$-invariant minimal disk $\Omega_i\subset \H^3$ so that $\Omega_i/E_*(G_i)=\tilde{\Sigma}_i,$ where $G_i<\Gamma_M$ is the group induced by $\Sigma_i\rightarrow M$. 
Since  the sequence $\{\Pi_i\}_{i\in\N}$ becomes equidistributed, we  proceed as in the proof of  Corollary 6.2
of \cite{calegari-marques-neves} to conclude the existence of $\phi_i\in \Gamma_M$  such that, after passing to a subsequence, it follows that
\begin{itemize}
\item[(a)] $\Lambda(\phi_iG_i\phi_i^{-1})$ converges  in Hausdorff distance, as $i\to\infty$, to $\sigma\in QC_0$;
\item[(b)]  $\{\phi(\sigma):\phi\in\Gamma_M\}$ is dense in $QC_0$;
\item[(c)]  for any $R>0$,
$$\lim_{i\to\infty}\int_{E_*(\phi_i)(\Omega_i)\cap B_R(0)}|A|^2dA_{\overline{g}_N}=0.$$
\end{itemize}

The sequence $E_*(\phi_i)(\Omega_i)$ has bounded curvature and hence converges  on compact sets, after passing to a subsequence, to a totally geodesic disk  $\Omega\subset \H^3$ such that $\partial_{\infty}\Omega=\tilde{E}(\sigma)$, where $\sigma$ is the round circle of (a). Therefore $\tilde{E}(\sigma)$ is  a round  circle. Since $\tilde{E}\circ \phi=E_*(\phi)\circ \tilde{E}$ for any $\phi\in \Gamma_M$, and $E_*(\phi)$ acts on $S^2$ by a conformal transformation, it follows that $\tilde{E}(\phi(\sigma))$ is a round circle for any $\phi \in \Gamma_M$.  Using (b) it follows that $\tilde{E}$ maps any round circle to a round circle.

By Carath\'eodory  \cite{caratheodory} the map $\tilde{E}$ is a M\"obius transformation. Hence $\tilde{F}$ is a M\"obius transformation and there is a unique $G\in \text{Isom}(\H^3)$ so that $G_{|S^2}=\tilde{F}$. Therefore $G\circ\phi= F_*(\phi)\circ G$ for every $\phi\in\Gamma_N$, and hence the map $G$ induces a map $G:N\rightarrow M$ which is  an isometry  homotopic to $F$. This finishes the proof of Theorem \ref{mostow}.
\end{proof}

\medskip

We now argue using Mostow's Rigidity Theorem that the entropies $E,E_{\overline{\nu}}$ of a Riemannian metric $g$ on $M$ do not depend 
on the hyperbolic metric. Let $\overline{g}_1$, $\overline{g}_2$ be hyperbolic metrics on $M$. Let $F:(M,\overline{g}_1)\rightarrow (M,\overline{g}_2)$ be an isometry that is homotopic to the identity map. If $\Pi$ is a homotopy class of surfaces in $M$ and $\Sigma\in \Pi$, then this implies  $F(\Sigma)\in \Pi$. Since ${\rm area}_{\overline{g}_2}(F(\Sigma))={\rm area}_{\overline{g}_1}(\Sigma),$ it follows that ${\rm area}_{\overline{g}_2}(\Pi)={\rm area}_{\overline{g}_1}(\Pi)$. 

The restriction $\tilde{F}:S^2 \rightarrow S^2$ to $S^2$ of the lift of $F$ to $\H^3$ is a conformal map. Hence 
the homotopy class $\Pi$ is $(1+\eta)$-quasi-Fuchsian for $\overline{g}_1$ if and only if it is $(1+\eta)$-quasi-Fuchsian for $\overline{g}_2$. Therefore, by Proposition \ref{alternative.definition},  $E(g)$ does not depend on whether the hyperbolic metric is $\overline{g}_1$ or $\overline{g}_2$. 
The Liouville measures on $Gr_2(M)$ satisfy $F_*(\overline{\nu}_1)=\overline{\nu}_2$, where we are denoting by $F$ also the induced action on $Gr_2(M)$. 
Hence   by (\ref{alternative.entropy.mu}) this is true for  $E_{\overline{\nu}}(g)$.

\bibliographystyle{amsbook}

  \end{document}